\documentclass[a4paper]{amsart}

\usepackage[dvipsnames]{xcolor}
\usepackage{amssymb,mathtools}
\usepackage{bbold}
\usepackage{tikz,tikz-cd,tikz-3dplot}
	\usetikzlibrary{positioning,calc}
\usepackage{enumitem}

\usepackage{hyperref}
\hypersetup{colorlinks=true, citecolor=Blue, linkcolor=Blue}
\usepackage[capitalise]{cleveref}
\usepackage{cancel}

\numberwithin{equation}{section}

\newtheorem{theorem}{Theorem}[section]
\newtheorem{lemma}[theorem]{Lemma}
\newtheorem{proposition}[theorem]{Proposition}

\newtheorem*{theorem*}{Theorem}

\theoremstyle{definition}
\newtheorem{definition}[theorem]{Definition}
\newtheorem{construction}[theorem]{Construction}
\newtheorem*{construction*}{Construction}
\newtheorem{example}[theorem]{Example}

\theoremstyle{remark}
\newtheorem{remark}[theorem]{Remark}

\definecolor{darkblue}{rgb}{0,0,0.7}
\definecolor{green}{RGB}{57,181,74}
\definecolor{violet}{RGB}{147,39,143}

\newcommand{\darkblue}{\color{darkblue}}

\newcommand{\bruhat}[2]{\mathcal{B}(#1, #2)}
\DeclareMathOperator{\inv}{inv}
\newcommand{\rhbo}[3]{\mathcal{B}_{#1}(#2, #3)}
\newcommand{\rinv}[2]{\inv_{#1}(#2)}
\newcommand{\packet}[1]{\mathsf{Pack}(#1)}

\DeclareMathOperator{\Hom}{Hom}

\DeclareMathOperator{\Sq}{Sq}

\DeclareMathOperator{\chains}{C}

\newcommand{\bd}{\partial}
\newcommand{\xla}[1]{\xleftarrow{#1}}

\newcommand{\N}{\mathbb{N}}
\newcommand{\Z}{\mathbb{Z}}

\newcommand{\ZZ}{\Z / 2\Z}

\newcommand{\op}{\mathrm{op}}

\newcommand{\defn}[1]{{\darkblue \emph{#1}}}

\newcommand{\simp}{\mathbb{\Delta}}
\newcommand{\st}{\mid}
\renewcommand{\emptyset}{\varnothing}

\newcommand{\init}[2]{A^{#1}_{#2}}
\newcommand{\initb}[2]{B^{#1}_{#2}}

\newcommand{\trmsgn}{\varepsilon}

\newcommand{\floor}[1]{\lfloor #1 \rfloor}
\newcommand{\ceil}[1]{\lceil #1 \rceil}

\newcommand{\dg}[1]{\mathsf{deg}(#1)}
\newcommand{\sz}[1]{|#1|}

\usepackage[backend=biber,
	style=alphabetic,
	url=false,
	doi=false,
	isbn=false,
	maxbibnames=99]{biblatex}
\DeclareFieldFormat{title}{#1}
\DeclareFieldFormat{postnote}{#1}
\renewbibmacro{in:}{}
\makeatletter
\def\@tocline#1#2#3#4#5#6#7{\relax
  \ifnum #1>\c@tocdepth 
  \else
    \par \addpenalty\@secpenalty\addvspace{#2}%
    \begingroup \hyphenpenalty\@M
    \@ifempty{#4}{%
      \@tempdima\csname r@tocindent\number#1\endcsname\relax
    }{%
      \@tempdima#4\relax
    }%
    \parindent\z@ \leftskip#3\relax \advance\leftskip\@tempdima\relax
    \rightskip\@pnumwidth plus4em \parfillskip-\@pnumwidth
    #5\leavevmode\hskip-\@tempdima
      \ifcase #1
       \or\or \hskip 1em \or \hskip 2em \else \hskip 3em \fi%
      #6\nobreak\relax
    \dotfill\hbox to\@pnumwidth{\@tocpagenum{#7}}\par
    \nobreak
    \endgroup
  \fi}
\makeatother

\title[Steenrod operations via higher Bruhat orders]{Steenrod operations \\ via higher Bruhat orders}

\author{Guillaume Laplante-Anfossi}
\address{Centre for Quantum Mathematics, Syddansk Universitet, Campusvej 55, Odense, Denmark}
\email{glaplanteanfossi@imada.sdu.dk}
\urladdr{https://guillaumelaplante-anfossi.github.io/}

\author{Nicholas J. Williams}
\address{Department of Pure Mathematics and Mathematical Statistics, Centre for Mathematical Sciences, University of Cambridge, Wilberforce Road, Cambridge, CB3 0WB, United Kingdom}
\address{Theoretical Sciences Visiting Program, Okinawa Institute of Science and Technology Graduate University, Onna, 904-0495, Japan}
\email{nw480@cam.ac.uk}
\urladdr{https://nchlswllms.github.io/}

\subjclass[2020]{55U15, 55S10, 52C22}

\begin{document}

\begin{abstract}
	The purpose of this paper is to establish a correspondence between the higher Bruhat orders of Yu.\ I. Manin and V. Schechtman, and the cup-$i$ coproducts defining Steenrod squares in cohomology.
	To any element of the higher Bruhat orders we associate a coproduct, recovering Steenrod's original ones from extremal elements in these orders. 
	Defining this correspondence involves interpreting the coproducts geometrically in terms of zonotopal tilings, which allows us to give conceptual proofs of their properties and show that all reasonable coproducts arise from our construction.
\end{abstract}

\thanks{GLA was supported by the Andrew Sisson Fund and the Australian Research Council Future Fellowship FT210100256, while NJW was supported by EPSRC grant EP/V050524/1 and is currently supported by EPSRC grant EP/W001780/1.}

\maketitle

\begin{figure}[h!]
\centerline{
\resizebox{0.75\linewidth}{!}{
\begin{tikzpicture}[xscale=1.8,scale=2,yscale=1.5]


\coordinate (e) at (0,0);
\coordinate (0) at (-1,0.75);
\coordinate (01) at (-1.5,2);
\coordinate (012) at (-1.5,3);
\coordinate (0123) at (-1,4.25);
\coordinate (01234) at (0,5);
\coordinate (1234) at (1,4.25);
\coordinate (234) at (1.5,3);
\coordinate (34) at (1.5,2);
\coordinate (4) at (1,0.75);

\draw (e) -- (0) -- (01) -- (012) -- (0123) -- (01234) -- (1234) -- (234) -- (34) -- (4) -- (e);


\coordinate (2) at (0,1);
\coordinate (02) at (-1,1.75);
\coordinate (23) at (0.5,2.25);
\coordinate (24) at (1,1.75);
\coordinate (023) at (-0.5,3);
\coordinate (123) at (0,3.5);


\draw (e) -- (2);
\draw (2) -- (02);
\draw (2) -- (23);
\draw (2) -- (24);
\draw (0) -- (02);
\draw (4) -- (24);
\draw (02) -- (012);
\draw (02) -- (023);
\draw (23) -- (023);
\draw (23) -- (123);
\draw (23) -- (234);
\draw (24) -- (234);
\draw (023) -- (0123);
\draw (123) -- (0123);
\draw (123) -- (1234);


\node at (e) [below = 1mm of e] {\tiny $\emptyset$};
\node at (01234) [above = 1mm of 01234] {\tiny $01234$};

\node at (0) [left = 1mm of 0] {\tiny $0$};
\node at (01) [left = 1mm of 01] {\tiny $01$};
\node at (012) [left = 1mm of 012] {\tiny $012$};
\node at (0123) [left = 1mm of 0123] {\tiny $0123$};

\node at (4) [right = 1mm of 4] {\tiny $4$};
\node at (34) [right = 1mm of 34] {\tiny $34$};
\node at (234) [right = 1mm of 234] {\tiny $234$};
\node at (1234) [right = 1mm of 1234] {\tiny $1234$};

\node at (2) [left = 1mm of 2] {\tiny $2$};
\node at (02) [right = 1mm of 02] {\tiny $02$};
\node at (24) [left = 1mm of 24] {\tiny $24$};
\node at (23) [left = 1mm of 23] {\tiny $23$};
\node at (023) [below = 3mm of 023] {\tiny $023$};
\node at (123) [above = 1mm of 123] {\tiny $123$};


\node at ($(e)!0.5!(02)$) {\color{red} $-02 \otimes 01234$};
\node at ($(e)!0.5!(24)$) {\color{red} $-24 \otimes 01234$};
\node [rotate=-65] at ($(0)!0.5!(012)$) {\color{red} $+012 \otimes 1234$};
\node at ($(2)!0.5!(023)$) {\color{red} $-023 \otimes 0134$};
\node [rotate=65] at ($(4)!0.5!(234)$) {\color{red} $+234 \otimes 0123$};
\node [rotate=30] at ($(2)!0.5!(234)$) {\color{red} $+234 \otimes 0134$};
\node at ($(02)!0.5!(0123)$) {\color{red} $-0123 \otimes 134$};
\node [rotate=-30] at ($(23)!0.5!(0123)$) {\color{red} $-0123 \otimes 014$};
\node at ($(23)!0.5!(1234)$) {\color{red} $-1234 \otimes 014$};
\node at ($(123)!0.5!(01234)$) {\color{red} $+01234 \otimes 04$};


\node [rotate=-30] at ($(e)!0.5!(0)$) [below = 1mm of $(e)!0.5!(0)$] {\color{blue} \tiny $+0 \otimes 01234$};
\node [rotate=-65] at ($(0)!0.5!(01)$) [below = 1mm of $(0)!0.5!(01)$] {\color{blue} \tiny $+01 \otimes 1234$};
\node [rotate=90] at ($(01)!0.5!(012)$) [above = 1mm of $(01)!0.5!(012)$] {\color{blue} \tiny $+012 \otimes 234$};
\node [rotate=65] at ($(012)!0.5!(0123)$) [above = 1mm of $(012)!0.5!(0123)$] {\color{blue} \tiny $+0123 \otimes 34$};
\node [rotate=30] at ($(0123)!0.5!(01234)$) [above = 1mm of $(0123)!0.5!(01234)$] {\color{blue} \tiny $+01234 \otimes 4$};

\node [rotate=30] at ($(e)!0.5!(4)$) [below = 1mm of $(e)!0.5!(4)$] {\color{blue} \tiny $-4 \otimes 01234$};
\node [rotate=65] at ($(4)!0.5!(34)$) [below = 1mm of $(4)!0.5!(34)$] {\color{blue} \tiny $+34 \otimes 0123$};
\node [rotate=-90] at ($(34)!0.5!(234)$) [above = 1mm of $(34)!0.5!(234)$] {\color{blue} \tiny $-234 \otimes 012$};
\node [rotate=-65] at ($(234)!0.5!(1234)$) [above = 1mm of $(234)!0.5!(1234)$] {\color{blue} \tiny $+1234 \otimes 01$};
\node [rotate=-30] at ($(1234)!0.5!(01234)$) [above = 1mm of $(1234)!0.5!(01234)$] {\color{blue} \tiny $-01234 \otimes 0$};


\node [rotate=90] at ($(e)!0.5!(2)$) [above = 0.1mm of $(e)!0.5!(2)$] {\color{OliveGreen} \tiny $-2 \otimes 01234$};
\node [rotate=-90] at ($(e)!0.5!(2)$) [above = 0.1mm of $(e)!0.5!(2)$] {\color{OliveGreen} \tiny $+2 \otimes 01234$};

\node [rotate=90] at ($(0)!0.5!(02)$) [above = 0.1mm of $(0)!0.5!(02)$] {\color{OliveGreen} \tiny $-02 \otimes 1234$};
\node [rotate=90] at ($(0)!0.5!(02)$) [below = 0.1mm of $(0)!0.5!(02)$] {\color{OliveGreen} \tiny $+02 \otimes 1234$};

\node [rotate=-30] at ($(2)!0.5!(02)$) [above = 0.1mm of $(2)!0.5!(02)$] {\color{OliveGreen} \tiny $-02 \otimes 0134$};
\node [rotate=-30] at ($(2)!0.5!(02)$) [below = 0.1mm of $(2)!0.5!(02)$] {\color{OliveGreen} \tiny $+02 \otimes 0134$};

\node [rotate=30] at ($(2)!0.5!(24)$) [above = 0.1mm of $(2)!0.5!(24)$] {\color{OliveGreen} \tiny $-24 \otimes 0134$};
\node [rotate=30] at ($(2)!0.5!(24)$) [below = 0.1mm of $(2)!0.5!(24)$] {\color{OliveGreen} \tiny $+24 \otimes 0134$};

\node [rotate=-90] at ($(4)!0.5!(24)$) [above = 0.1mm of $(4)!0.5!(24)$] {\color{OliveGreen} \tiny $-24 \otimes 0123$};
\node [rotate=-90] at ($(4)!0.5!(24)$) [below = 0.1mm of $(4)!0.5!(24)$] {\color{OliveGreen} \tiny $+24 \otimes 0123$};

\node [rotate=65] at ($(24)!0.5!(234)$) [above = 0.1mm of $(24)!0.5!(234)$] {\color{OliveGreen} \tiny $-234 \otimes 013$};
\node [rotate=65] at ($(24)!0.5!(234)$) [below = 0.1mm of $(24)!0.5!(234)$] {\color{OliveGreen} \tiny $+234 \otimes 013$};

\node [rotate=65] at ($(2)!0.5!(23)$) [above = 0.1mm of $(2)!0.5!(23)$] {\color{OliveGreen} \tiny $-23 \otimes 0134$};
\node [rotate=65] at ($(2)!0.5!(23)$) [below = 0.1mm of $(2)!0.5!(23)$] {\color{OliveGreen} \tiny $+23 \otimes 0134$};

\node [rotate=65] at ($(02)!0.5!(023)$) [above = 0.1mm of $(02)!0.5!(023)$] {\color{OliveGreen} \tiny $+023 \otimes 134$};
\node [rotate=65] at ($(02)!0.5!(023)$) [below = 0.1mm of $(02)!0.5!(023)$] {\color{OliveGreen} \tiny $-023 \otimes 134$};

\node [rotate=-30] at ($(23)!0.5!(023)$) [above = 0.1mm of $(23)!0.5!(023)$] {\color{OliveGreen} \tiny $+023 \otimes 014$};
\node [rotate=-30] at ($(23)!0.5!(023)$) [below = 0.1mm of $(23)!0.5!(023)$] {\color{OliveGreen} \tiny $-023 \otimes 014$};

\node [rotate=-65] at ($(02)!0.5!(012)$) [above = 0.1mm of $(02)!0.5!(012)$] {\color{OliveGreen} \tiny $+012 \otimes 134$};
\node [rotate=-65] at ($(02)!0.5!(012)$) [below = 0.1mm of $(02)!0.5!(012)$] {\color{OliveGreen} \tiny $-012 \otimes 134$};

\node [rotate=30] at ($(23)!0.5!(234)$) [above = 0.1mm of $(23)!0.5!(234)$] {\color{OliveGreen} \tiny $-234 \otimes 014$};
\node [rotate=30] at ($(23)!0.5!(234)$) [below = 0.1mm of $(23)!0.5!(234)$] {\color{OliveGreen} \tiny $+234 \otimes 014$};

\node [rotate=-65] at ($(023)!0.5!(0123)$) [above = 0.1mm of $(023)!0.5!(0123)$] {\color{OliveGreen} \tiny $+0123 \otimes 14$};
\node [rotate=-65] at ($(023)!0.5!(0123)$) [below = 0.1mm of $(023)!0.5!(0123)$] {\color{OliveGreen} \tiny $-0123 \otimes 14$};

\node [rotate=-65] at ($(23)!0.5!(123)$) [above = 0.1mm of $(23)!0.5!(123)$] {\color{OliveGreen} \tiny $+123 \otimes 014$};
\node [rotate=-65] at ($(23)!0.5!(123)$) [below = 0.1mm of $(23)!0.5!(123)$] {\color{OliveGreen} \tiny $-123 \otimes 014$};

\node [rotate=30] at ($(123)!0.5!(1234)$) [above = 0.1mm of $(123)!0.5!(1234)$] {\color{OliveGreen} \tiny $+1234 \otimes 04$};
\node [rotate=30] at ($(123)!0.5!(1234)$) [below = 0.1mm of $(123)!0.5!(1234)$] {\color{OliveGreen} \tiny $-1234 \otimes 04$};

\node [rotate=-30] at ($(123)!0.5!(0123)$) [above = 0.1mm of $(123)!0.5!(0123)$] {\color{OliveGreen} \tiny $+0123 \otimes 04$};
\node [rotate=-30] at ($(123)!0.5!(0123)$) [below = 0.1mm of $(123)!0.5!(0123)$] {\color{OliveGreen} \tiny $-0123 \otimes 04$};

\end{tikzpicture}}
}
\label{fig:big_ex}
\end{figure}


\tableofcontents


\section{Introduction}

Steenrod operations, introduced by N. E. Steenrod in \cite{s47}, are invariants refining the algebra structure given by the cup product on the cohomology of a space.
They are defined via a family of cup-$i$ coproducts, which correct homotopically the lack of cocommutativity of the Alexander--Whitney diagonal at the chain level. 
The more refined homotopical information provided by Steenrod operations allows one to distinguish non-homotopy equivalent spaces with isomorphic cohomology rings (for example, the suspensions of $\mathbb{C}P^2$ and $S^2 \vee S^4$).
Steenrod operations are universal and constitute a central, classical tool in homotopy theory \cite{SteenrodEpstein63,MosherTangora}.
More recently, they were shown to be part of an $E_\infty$-algebra structure on the cochains of a space $X$ \cite{ms03,BergerFresse04}, which encodes faithfully its homotopy type when $X$ is of finite type and nilpotent \cite{Mandell01,Mandell06}. 

On the other hand, the higher Bruhat orders are a family of posets introduced by Yu.\ I. Manin and V.~Schechtman \cite{ms89} generalising the weak Bruhat order on the symmetric group.
The elements of the $(n+1)$-th higher Bruhat order are equivalence classes of maximal chains in the $n$-th higher Bruhat order, with covering relations given by higher braid moves.
The original motivation for introducing these orders was to study hyperplane arrangements and generalised braid groups, but they have subsequently found applications in many different areas.
They appear in the contexts of Soergel bimodules \cite{e16}, quantisations of the homogeneous coordinate ring of the Grassmannian \cite{lz98}, and KP solitons \cite{dm12}.
The higher Bruhat orders also provide a framework for studying social choice in economics \cite{gr08}.

In this paper, we show that these two apparently very different objects are in fact combinatorially the same: the higher Bruhat orders describe precisely Steenrod cup-$i$ coproducts and their relations. 
This provides a possible explanation for the fact that, in the words of A. Medina-Mardones, ``cup-$i$ products of Steenrod seem to be combinatorially fundamental'' \cite[Sec.~3.3]{m21}.
Let $\bruhat{[0,n]}{i + 1}$ be the $(i + 1)$-dimensional higher Bruhat order on the set $[0,n]$, and let $\Delta_{i}, \Delta_{i}^\op  \colon \chains_{\bullet}(\simp^n) \to \chains_{\bullet}(\simp^n) \otimes \chains_{\bullet}(\simp^n)$ denote the Steenrod cup-$i$ coproduct and its opposite on the chain complex of the standard simplex. 

\begin{theorem*}[{\cref{const}, \cref{lem:boundary=steenrod,thm:homotopy_formula,thm:all_coproducts}}]
\label{thm:int:main}
For every element $U \in \bruhat{[0, n]}{i + 1}$, there is a coproduct \[\Delta_{i}^{U}\colon \chains_{\bullet}(\simp^n) \to \chains_{\bullet}(\simp^n) \otimes \chains_{\bullet}(\simp^n)\] which gives a homotopy between $\Delta_{i - 1}$ and $\Delta_{i - 1}^\op$.
 If $U_{\min}$ and $U_{\max}$ are the maximal and minimal elements of $\bruhat{[0,n]}{i + 1}$, then $\{\Delta_{i}^{U_{\min}}, \Delta_{i}^{U_{\max}}\} = \{\Delta_{i}, \Delta_{i}^{\mathrm{op}}\}$.
 
Moreover, every coproduct on $\chains_\bullet(\simp^n)$ giving a homotopy between $\Delta_{i - 1}$ and $\Delta_{i - 1}^{\mathrm{op}}$ arises in this way, so long as it does not contain redundant terms.
\end{theorem*}

From the perspective of the higher Bruhat orders, the fact that every coproduct $\Delta_{i}^{U}$ gives a homotopy between $\Delta_{i - 1}$ and $\Delta_{i - 1}^{\op}$ corresponds to the fact that the elements of $\bruhat{[0, n]}{i + 1}$ are equivalence classes of maximal chains in $\bruhat{[0, n]}{i}$.
In fact, from any covering relation $U \lessdot V$ in $\bruhat{[0,n]}{i + 1}$, one can construct a chain homotopy between $\Delta_i^U$ and $\Delta_i^V$ (\cref{const:cov_rel}).
Moreover, as we demonstrate in \cref{sect:squares}, any coproduct $\Delta_i^U$ defines a Steenrod square $\Sq_i^U$ in cohomology, and for any two $U,V \in \bruhat{[0,n]}{i + 1}$ we have $\Sq_i^U=\Sq_i^V$ (\cref{thm:Steenrod-squares}).

Using the geometric interpretation of the elements of the higher Bruhat orders as \emph{zonotopal tilings} (\cref{sect:back:hbo}), we obtain a clear geometric interpretation of homotopies (see the front page, as well as \cref{fig:z32_lower_upper}): terms in the image of the coproducts correspond to faces of tilings, and the boundary map on $\Hom(\chains_\bullet(\simp^n),\chains_\bullet(\simp^n)\otimes \chains_\bullet(\simp^n))$ corresponds to the cubical boundary map.
This gives a conceptual proof of the fundamental property that the cup-$i$ coproduct gives a homotopy between the cup-$(i - 1)$ coproduct and its opposite, in contrast to the involved combinatorial proofs in the literature.

The above results allow us to show that for homology of simplicial complexes non-trivial coproducts from other elements of the higher Bruhat orders exist (\cref{sect:imp:simp_comp}), whereas for singular homology only the Steenrod coproducts are possible (\cref{sect:imp:sing}).
We also show how one can find coproducts giving homotopies between $\Delta_{i}^U$ and its opposite, using the ``reoriented" higher Bruhat orders of \cite{z93,fw00} (\cref{sect:imp:reorient}).
In his recent axiomatic characterisation of Steenrod's cup-$i$ products,
A. Medina-Mardones shows that, under a certain natural notion of isomorphism, there is only one cup-$i$ construction \cite{m22}. 
Our cup-$i$ constructions $\Delta_i^U$ provide other choices, which fall outside the scope of this result, but are still adequate for the purpose of defining Steenrod squares in cohomology (\cref{sect:squares}).
New formulas for Steenrod squares could be of interest from the computational point of view \cite{medina-mardones23}, notably in the field of Khovanov homology \cite{canteromoran20}.
At the same time, our unicity result for singular homology (\cref{thm:sing}) concurs with the one of \cite{m22}.

A consequence of the present results is that there is a dictionary between the Steenrod coproducts and two important families of objects in other areas of mathematics, which are already known to correspond to higher Bruhat orders. 
The first one is a family of higher dimensional versions of the Yang--Baxter (or ``triangle'') equation, called the \emph{simplex equations} \cite{s95}, which have been a subject of renewed interest in the recent physics literature \cite{BazhanovSixvertex,KunibaTetrahedron,YagiIntegrable}.
The second one is a family of strict $\omega$-categories called \emph{cubical orientals}, which define the cubical $\omega$-categorical nerve \cite{s91}.
The higher categorical structure of the higher Bruhat orders was already observed in \cite[Sec.~3]{ms89}, and their correspondence with cubical orientals can be found in \cite{kv91}, see also \cite{lmp}.

Finally, it seems that the cubical structure of the original Steenrod cup-$i$ products is already present in the topological $E_\infty$-operad defined in \cite{Kaufmann09}, see Figure~11 therein. 
It would also be interesting to understand the possible connection between our construction and the permutahedral structures on $E_2$-operads unravelled in~\cite{KaufmannZhang17}.

\subsection*{Outline}

We begin the paper by giving background on Steenrod operations in Section~\ref{sect:back:steenrod} and on the higher Bruhat orders in Section~\ref{sect:back:hbo}.
In Section~\ref{sect:main}, we give our main construction, showing how to build coproducts from elements of the higher Bruhat orders, proving their key properties, and showing that all reasonable coproducts arise in this way.
After this, in Section~\ref{sect:imp}, we give some extensions of our construction in the previous section, discussing how it impacts simplicial and singular cohomology, and extending it to the reoriented higher Bruhat orders.
We finish by showing that the coproducts we construct all induce the same Steenrod squares.

\subsection*{Acknowledgements}

We would first and foremost like to thank Hugh Thomas for making the introduction that led to this collaboration.
The initial discussions for this paper took place whilst GLA was visiting the Max Planck Institute in Bonn, which we thank for the hospitality.
Thank you also to Arun Ram for helpful suggestions.
GLA thanks Anibal Medina-Mardones for useful discussions on Steenrod squares, as well as Fabian Haiden and Vivek Shende for the opportunity to speak about a preliminary version of this work in Odense, and for interesting discussions.
Part of this research was conducted while NJW was visiting the Okinawa Institute of Science and Technology (OIST) through the Theoretical Sciences Visiting Program (TSVP).


\section{Background}
\label{sect:back}

In this section, we recall the definitions of Steenrod operations and higher Bruhat orders, and set up notation. 

\subsection{Steenrod operations}
\label{sect:back:steenrod}

We start by recalling basic conventions and notation. 


\subsubsection{Chain complexes}
\label{sect:back:at_conv}

By a \defn{chain complex} $C$ we mean an $\N$-graded $\mathbb{Z}$-module with linear maps
\[
C_0 \xla{\bd_1} C_1 \xla{\bd_2} C_2 \xla{\bd_3} \dotsb
\]
satisfying $\bd_p \circ \bd_{p+1} = 0$ for each $p \in \N$.
As usual, we refer to $\bd_p$ as the $p$-th \defn{boundary map} and suppress the subscript when convenient.
A degree-$i$ \defn{morphism of chain complexes} $f \colon C \to C'$, referred to as a degree-$i$ \defn{chain map}, is a set of morphisms of $\mathbb{Z}$-modules $f_{p} \colon C_{p} \to C'_{p + i}$ satisfying $\bd'_{p+i+1} f_{p+1} = f_p \bd_{p+1}$ for $p \in \N$.

The category of chain complexes of $\mathbb{Z}$-modules is endowed with a tensor product, whose degree-$r$ component is  $(X \otimes Y)_r := \oplus_{p + q = r} X_p \otimes Y_q$, and whose differential is defined by $\partial(x \otimes y) := \partial(x)\otimes y + (-1)^{\dg{x}} x \otimes \partial(y)$. 
The symmetry isomorphism~$T \colon X \otimes Y \to Y \otimes X$, which is part of a symmetric monoidal structure, is defined by $T(x \otimes y):=(-1)^{\dg{x}\dg{y}}y \otimes x$.


\subsubsection{Steenrod coalgebra}\label{sect:back:steenrod_co}

We denote the standard $n$-simplex $\simp^n := \{(x_0, \ldots, x_n) \in \mathbb{R}^{n + 1} \st x_0 + \cdots + x_{n} = 1, x_{i} \geqslant 0 \}$.
We refer to faces of the $n$-simplex using their vertex sets, where we use the notation $[p, q] := \{p, p + 1, \dots, q\}$ and $(p, q) := [p, q] \setminus \{p, q\}$.
When we give a set $\{v_{0}, v_{1}, \dots, v_{q}\} \subseteq [0, n]$, we mean that the elements are ordered $v_{0} < v_{1} < \dots < v_{q}$.

We will consider the $\mathbb{Z}$-module given by the cellular chains $\chains_\bullet(\simp^n)$ on the standard $n$-simplex. 
This chain complex has as basis the faces of $\simp^n$, whose degree is given by the dimension (for example, the face $\{v_0, \ldots, v_q\}$ has dimension $q$). 
The boundary map of this chain complex is given by \[ \partial(\{v_0, \ldots, v_q \}) := \sum_{p = 0}^{q}(-1)^{p}\{v_0,\ldots,\hat v_p, \ldots, v_q\} \ . \]

An \defn{overlapping partition} $\mathcal{L} = (L_{0}, L_{1}, \dots, L_{i+1})$ of $[0, n]$ is a family of intervals $L_{p} = [l_{p}, l_{p+1}]$ such that $l_{0} = 0$, $l_{i+2} = n$, and for each $0 < p < i+1$ we have $l_{p} < l_{p+1}$. 
For $i \geqslant 0$, the \defn{Steenrod cup-$i$ coproduct} is the degree-$i$ linear map $\Delta_i \colon \chains_\bullet(\simp^n) \to \chains_\bullet(\simp^n) \otimes \chains_\bullet(\simp^n)$ defined on the top face by 
\[\Delta_i([0, n]) := \sum_{\mathcal{L}} (-1)^{\varepsilon(\mathcal{L})} (L_{0} \cup L_{2} \cup \cdots) \otimes (L_{1} \cup L_{3} \cup \cdots) \ , \] 
where the sum is taken over all overlapping partitions of $[0, n]$ into $i + 2$ intervals.
By convention, one sets $\Delta_{-1}:=0$.
If $n \leqslant i - 1$, there are no such overlapping partitions, and the coproduct is zero.
Denoting by $w_{\mathcal{L}}$ the shuffle permutation putting $0,1, \dots, n$ into the order \[[0, l_{1}], [l_{2}, l_{3}], \dots, (l_{1}, l_{2}), (l_{3}, l_{4}), \dots,\] the sign is given by~$\varepsilon(\mathcal{L}) := \mathsf{sign}(w_{\mathcal{L}}) + in$.
The coproduct $\Delta_i$ is then defined similarly on the lower-dimensional faces.
Throughout this article, we shall reserve the notation $\Delta_i$ for the Steenrod coproducts.  
As proved by N. E. Steenrod in \cite{s47}, the cup-$i$ coproducts satisfy the \defn{homotopy formula}
\begin{equation}
	\label{eq:homotopy}
	\bd \Delta_i -(-1)^i \Delta_i \bd = (1+(-1)^iT)\Delta_{i-1}   
\end{equation}
for all $i \geqslant 0$.
One can say that $\Delta_{i}$ measures the obstruction to $\Delta_{i - 1}$ being cocommutative.

\begin{remark}
Other equivalent formulas for the Steenrod coproducts are given in \cite{gr99,m23}.
\end{remark}

\begin{example}\label{ex:back:steenrod}
We give some examples of Steenrod cup-$i$ coproducts for low-dimensional simplices.
For the $0$-simplex $\simp^0$, we have $\Delta_0(0)=T\Delta_0(0)=0 \otimes 0$. 
For the $1$-simplex $\simp^1$, we have
\begin{align*}
\Delta_0(01) &= 0 \otimes 01 + 01 \otimes 1 \ , \\
	T\Delta_0(01) &= 01 \otimes 0 + 1 \otimes 01 \ , \\
	\Delta_1(01)=-T\Delta_1(01) &= - 01 \otimes 01 \ .
\end{align*}
Finally, for the $2$-simplex $\simp^2$, we have
\begin{align*}
\Delta_0(012) &= 0 \otimes 012 + 01 \otimes 12 + 012 \otimes 2 \ , \\
	T\Delta_0(012) &= 012 \otimes 0 - 12 \otimes 01 + 2 \otimes 012 \ , \\
	\Delta_1(012) &= 012 \otimes 01 - 02 \otimes 012 + 012 \otimes 12 \ , \\
	T\Delta_1(012) &= 01 \otimes 012 - 012 \otimes 02 + 12 \otimes 012 \ , \\
	\Delta_2(012)=T\Delta_2(012) &= 012 \otimes 012 \ .
\end{align*}
\end{example}


\subsection{Higher Bruhat orders}\label{sect:back:hbo}

There are many ways of defining the higher Bruhat orders.
For the purposes of this paper, it will be useful to consider three of them, namely those using admissible orders, consistent sets, and cubillages.
The cubillage perspective is the principal one we use for our construction.


\subsubsection{Admissible orders}\label{sect:back:hbo:adm}

The original definition of the higher Bruhat orders from \cite{ms89} is as follows.
Let $i$ and $n$ be non-negative integers such that $i + 1 \leqslant n$.
Throughout this article we denote by $\binom{[0, n]}{i + 1}:=\{S \subset [0, n] \st \sz{S} = i + 1\}$ the set of $(i + 1)$-element subsets of $[0, n]$.
Given $K=\{k_{0} < k_{1} < \dots < k_{i + 1}\} \in \binom{[0, n]}{i + 2}$, the set \[\packet{K}:=\{K\setminus k \st k \in K \} \] is called the \defn{packet} of $K$.
Here we have abbreviated $K \setminus \{k\}$ to $K \setminus k$, which we will continue to do.
It is naturally ordered by the \defn{lexicographic order}, where $K \setminus k_{q} < K \setminus k_{p}$ if and only if $p < q$, or its opposite, the \defn{reverse lexicographic order}.

The elements of the \defn{higher Bruhat poset} $\bruhat{[0, n]}{i + 1}$ are admissible orders of $\binom{[0, n]}{i + 1}$, modulo an equivalence relation. 
A total order $\alpha$ of $\binom{[0, n]}{i + 1}$ is \defn{admissible} if for all $K \in \binom{[0, n]}{i + 2}$, the elements $\packet{K}$ appear in either lexicographic or reverse lexicographic order under~$\alpha$.
Two orderings $\alpha$ and $\alpha'$ of~$\binom{[0, n]}{i + 1}$ are \defn{equivalent} if they differ by a sequence of interchanges of pairs of adjacent elements that do not lie in a common packet.
As these interchanges preserve admissibility, the equivalence class $[\alpha]$ of an ordering $\alpha$ is well-defined.
 
The \defn{inversion set} $\inv(\alpha)$ of an admissible order $\alpha$ is the set of all $(i + 2)$-subsets of~$[0, n]$ whose packets appear in reverse lexicographic order in $\alpha$.
Note that inversion sets are well-defined on equivalence classes of admissible orders.
The poset structure on $\bruhat{[0, n]}{i + 1}$ is generated by the covering relations given by $[\alpha] \lessdot [\alpha']$ if $\inv(\alpha') = \inv(\alpha) \cup \{K\}$ for $K \in \binom{[0, n]}{i + 2} \setminus \inv(\alpha)$.
It is possible to have $\inv(\alpha) \subseteq \inv(\alpha')$ without having $[\alpha] \leqslant [\alpha']$ \cite[Thm.~4.5]{z93}.
For $n < i + 1$, we set $\bruhat{[0, n]}{i + 1} = \{\emptyset\}$.


\subsubsection{Consistent sets}\label{sect:back:hbo:cons}

An element $[\alpha]$ of the higher Bruhat poset $\bruhat{[0, n]}{i + 1}$ is uniquely determined by its inversion set $\inv(\alpha)$.
Inversion sets were characterised intrinsically in \cite{z93} as follows.
A subset $U \subseteq \binom{[0, n]}{i + 2}$ is \defn{consistent} if for any $M \in \binom{[0, n]}{i + 3}$, the intersection $\packet{M} \cap U$ is either a beginning segment of $\packet{M}$ in the lexicographic order or an ending segment.
We then have that a set $U \subseteq \binom{[0, n]}{i + 2}$ is equal to $\inv(\alpha)$ for some $[\alpha] \in \bruhat{[0, n]}{i + 1}$ if and only if $U$ is consistent \cite[Thm.~4.1]{z93}.
In this paper it will be convenient to work in terms of consistent sets, so we will abuse notation slightly by writing $U \in \bruhat{[0, n]}{i + 1}$ if $U$ is a consistent set.
The higher Bruhat order can then be defined on consistent sets by the covering relations $U \lessdot U'$ if and only if $U' = U \cup \{K\}$ for $K \in \binom{[0, n]}{i + 2} \setminus U$.

A maximal chain in $\bruhat{[0, n]}{i + 1}$ gives an order on $\binom{[0, n]}{i + 2}$ according to the order these subsets are added to the inversion set.
The consistency condition then ensures that this order is in fact admissible, and one can talk about equivalence of maximal chains in the same way as equivalence of admissible orders.
We then have the following theorem, which could justly be called the fundamental theorem of the higher Bruhat orders.

\begin{theorem}[{\cite[Thm.~2.3]{ms89}}]\label{thm:fund_hbo}
There is a bijection between equivalence classes of maximal chains in $\bruhat{[0, n]}{i + 1}$ and elements of $\bruhat{[0, n]}{i + 2}$.
\end{theorem}

We will finally need the following construction from \cite[Def.~7.4]{r97}, which we call \defn{contraction}, but which was called ``deletion" in \cite{r97}.
For $S \subseteq [0, n]$ and $U \in \bruhat{[0, n]}{i + 1}$, we denote by $U/S := U \cap \binom{[0, n] \setminus S}{i + 2}$.
We have that $U/S \in \bruhat{[0, n]\setminus S}{i + 1}$, where this poset is defined using the natural identification of $[0, n] \setminus S$ with $[0, n - \sz{S}]$.
When $S = \{p\}$, we write $U/p$ for $U/S$.


\subsubsection{Cubillages}
\label{sect:back:hbo:cub}

We now give the geometric description of the higher Bruhat orders due to \cite{kv91,t02}. 
Consider the \defn{Veronese curve} $\xi\colon \mathbb{R} \rightarrow \mathbb{R}^{i + 1}$, given by $\xi_{t}=(1, t, t^2, \dots, t^{i})$. Let $\left\lbrace t_{0}, \dots, t_{n}\right\rbrace \subset \mathbb{R}$ with $t_{0} < \dots < t_{n}$ and $n \geqslant i$. 
The \defn{cyclic zonotope} $Z([0, n], i + 1)$ is defined to be the Minkowski sum of the line segments \[\overline{\mathbf{0}\xi_{t_{0}}} + \dots + \overline{\mathbf{0}\xi_{t_{n}}} \ ,\] where $\mathbf{0}$ is the origin and $\overline{\mathbf{0}\xi_{t_{p}}}$ is the line segment from $\mathbf{0}$ to $\xi_{t_{p}}$. 
The properties of the zonotope do not depend on the exact choice of $\{t_{0}, \dots, t_{n}\} \subset \mathbb{R}$. 
Hence, for ease we set $t_{p} = p$ for all $p \in [0, n]$.

There is a natural projection $\pi_{i + 1} \colon Z([0, n], n + 1) \to Z([0, n], i + 1)$ given by forgetting the last $n - i$ coordinates.
A \defn{cubillage} $\mathcal{Q}$ of $Z([0, n], i + 1)$ is a section $\mathcal{Q} \colon Z([0, n], i + 1) \to Z([0, n], n + 1)$ of the projection $\pi_{i + 1} \colon Z([0, n], n + 1) \to Z([0, n], i + 1)$ whose image is a union of $(i + 1)$-dimensional faces of $Z([0, n], n + 1)$.
We call these $(i + 1)$-dimensional faces the \defn{cubes} of the cubillage.
A cubillage $\mathcal{Q}$ of $Z([0, n], i + 1)$ gives a subdivision of $Z([0, n], i + 1)$ consisting of the images of the projections of its cubes under~$\pi_{i + 1}$.
We usually think of the cubillage as the subdivision, but it is necessary to define the cubillage as the section to make the $i = 0$ case work, since for $i = 0$ all the subdivisions are the same; see \cite[Rem.~2.6]{gp21} for more details.
In the literature, cubillages are often called \defn{fine zonotopal tilings}, for example, in \cite{gp21}.

Recall that a \defn{facet} of a polytope is a face of codimension one.
The standard basis of $\mathbb{R}^{i + 1}$ induces orientations of the faces of $Z([0, n], i + 1)$, in the sense that the facets of a face can be partitioned into two sets, called upper facets and lower facets.
If $F$ is a $j$-dimensional face of $Z([0, n], i + 1)$, with $G$ a facet of $F$, then $G$ is a \defn{lower} (resp.\ \defn{upper}) facet of $F$ if any normal vector to $G$ which lies inside the affine span of $F$ and points into $F$ has a positive (resp.\ negative) $j$-th coordinate.

As was proven in \cite[Thm.~2.1, Prop.~2.1]{t02}, following \cite[Thm.~4.4]{kv91}, the elements of $\mathcal{B}([0, n], i + 1)$ are in bijection with cubillages of $Z([0, n], i + 1)$. 
The covering relations of $\mathcal{B}([0, n], i + 1)$ are given by pairs of cubillages $\mathcal{Q} \lessdot \mathcal{Q}'$ that differ by an \defn{increasing flip}, that is, when there is a $(i + 2)$-face $F$ of $Z([0, n], n + 1)$ such that $\mathcal{Q}(Z([0, n], i + 1)) \setminus F = \mathcal{Q}'(Z([0, n], i + 1)) \setminus F$ and $\mathcal{Q}(Z([0, n], i + 1))$ contains the lower facets of $F$, whereas $\mathcal{Q}'(Z([0, n], i + 1))$ contains the upper facets of~$F$. 

The cyclic zonotope $Z([0, n], i + 1)$ possesses two canonical cubillages.
One is given by the unique section $\mathcal{Q}_{l} \colon Z([0, n], i + 1) \to Z([0, n], n + 1)$ of $\pi_{i + 1}$ whose image projects to the union of the lower facets of $Z([0, n], i + 2)$ under~$\pi_{i + 2}$.
We call this the \defn{lower cubillage}.
The other is the section $\mathcal{Q}_{u}$ whose image projects to the upper facets of $Z([0, n], i + 2)$ under~$\pi_{i + 2}$, which we call the \defn{upper cubillage}. 
The lower cubillage of $Z([0, n], i + 1)$ gives the unique minimum of the poset $\mathcal{B}([0, n], i + 1)$, and the upper cubillage gives the unique maximum.
The lower and upper cubillages are also known as the \defn{standard} and \defn{antistandard} cubillages respectively, for instance, in \cite{dkk19}.

We now give a description of the cubes of a cubillage that will be useful later. 
Every $(i + 1)$-dimensional face of $Z([0, n], n + 1)$ is given by a Minkowski sum \[\xi_{A} + \sum_{l \in L} \overline{\mathbf{0}\xi_{l}}\] for some subset $L \in \binom{[0, n]}{i + 1}$ and $A \subseteq [0, n] \setminus L$, where $\xi_{A} = \sum_{a \in A} \xi_{a}$. 
We call $L$ the set of \defn{generating vectors} and $A$ the \defn{initial vertex}, see \cref{fig:iv_gv}.
Here we have omitted the $\xi$ from the labels of the vertices of the cubes, so that a label $A$ means~$\xi_{A}$; henceforth, we will always do this.

\begin{figure}[h!]
\centerline{
\resizebox{\linewidth}{!}{
\begin{tikzpicture}

\begin{scope}[shift={(-3,0.5)},scale=2]

\coordinate (1) at (0,0);
\coordinate (12) at (0,1);

\draw (1) -- (12);

\node at (1) [below = 1mm of 1] {\tiny $1$};
\node at (12) [above = 1mm of 12] {\tiny $12$};

\end{scope}


\begin{scope}[shift={(0,0.25)},scale=1.25]

\coordinate (3) at (0,0);
\coordinate (23) at (-1,1);
\coordinate (34) at (1,1);
\coordinate (234) at (0,2);

\draw (3) -- (23) -- (234) -- (34) -- (3);

\node at (3) [below = 1mm of 3] {\tiny $3$};
\node at (23) [left = 1mm of 23] {\tiny $23$};
\node at (34) [right = 1mm of 34] {\tiny $34$};
\node at (234) [above = 1mm of 234] {\tiny $234$};

\end{scope}


\begin{scope}[shift={(4,0)}]

\coordinate (12) at (0,0);
\coordinate (012) at (-1,1);
\coordinate (124) at (0,1);
\coordinate (125) at (1,1);
\coordinate (0124) at (-1,2);
\coordinate (0125) at (0,2);
\coordinate (1245) at (1,2);
\coordinate (01245) at (0,3);

\draw (12) -- (012) -- (0124) -- (01245) -- (1245) -- (125) -- (12);
\draw (12) -- (124);
\draw (124) -- (0124);
\draw (124) -- (1245);
\draw[dotted] (01245) -- (0125);
\draw[dotted] (0125) -- (125);
\draw[dotted] (0125) -- (012);

\node at (12) [below = 1mm of 12] {\tiny $12$};
\node at (012) [left = 1mm of 012] {\tiny $012$};
\node at (0124) [left = 1mm of 0124] {\tiny $0124$};
\node at (01245) [above = 1mm of 01245] {\tiny $01245$};
\node at (1245) [right = 1mm of 1245] {\tiny $1245$};
\node at (125) [right = 1mm of 125] {\tiny $125$};
\node at (124) [left = 1mm of 124] {\tiny $124$};
\node at (0125) [right = 1mm of 0125] {\tiny $0125$};

\end{scope}

\node [anchor=east] at (-3.75,-1) {\small Initial vertex};
\node [anchor=east] at (-3.75,-1.5) {\small Generating vectors};

\node at (-2.95,-1) {\small $\{1\}$};
\node at (-2.95,-1.5) {\small $\{2\}$};

\node at (0.05,-1) {\small $\{3\}$};
\node at (0.05,-1.5) {\small $\{2,4\}$};

\node at (4.1,-1) {\small $\{1,2\}$};
\node at (4.1,-1.5) {\small $\{0,4,5\}$};

\end{tikzpicture}}
}
\caption{Cubes associated to some vertices and sets of generating vectors.}
\label{fig:iv_gv}
\end{figure}
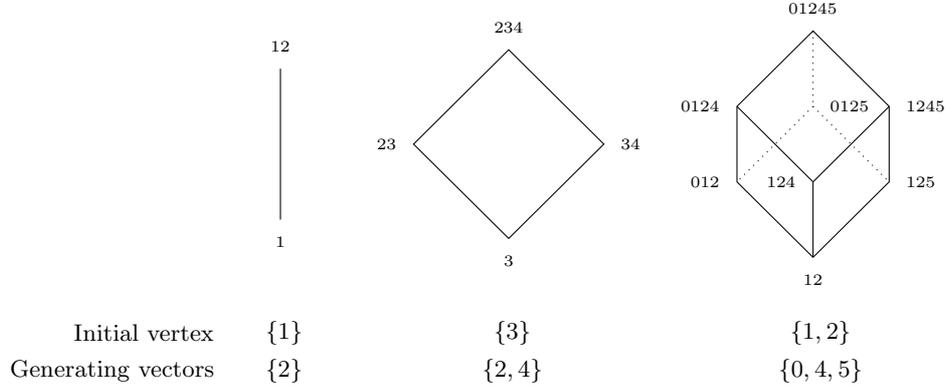

One can describe the upper and lower facets of a cube in terms of initial vertices and generating vectors using the following combinatorial notion.
Given $L \in \binom{[0, n]}{i + 1}$ and $a \in [0, n] \setminus L$, we will use the notation
\[\sz{L}_{>a} := \sz{\{l \in L \st l > a\}} \ .\]
We then say that $a$ is an \defn{even gap} if $\sz{L}_{>a}$ is an even number. 
Otherwise, we say that $a$ is an \defn{odd gap}.

\begin{proposition}[{\cite[Lem.~2.1]{t02} or \cite[Prop.~8.1]{dkk18}}]\label{prop:up_low_facets}
Let $F$ be a face of $Z([0, n], n + 1)$ with initial vertex $A$ and generating vectors $L$. Then we have the following.
\begin{enumerate}
\item Facets of $F$ with generating vectors $L \setminus l$ and initial vertex $A$ are upper facets if $l$ is an odd gap in $L \setminus l$ and lower facets if $l$ is an even gap in $L \setminus l$.
\item Facets of $F$ with generating vectors $L \setminus l$ and initial vertex $A \cup \{l\}$ are upper facets if $l$ is an even gap in $L \setminus l$ and lower facets if $l$ is an odd gap in $L \setminus l$.
\end{enumerate}
\end{proposition}


\subsubsection{From consistent sets to cubillages}
\label{sect:cons->cub}

One can explicitly describe the bijection between cubillages of $Z([0, n], i + 1)$ and consistent subsets of $\binom{[0, n]}{i + 2}$.
Given a cubillage $\mathcal{Q}$ of $Z([0, n], i + 1)$, it follows from \cite[Thm.~2.1]{t02} that the cubes of $\mathcal{Q}$ are in bijection with the elements of~$\binom{[0, n]}{i + 1}$ via sending a cube to its set of generating vectors.
Hence, given a consistent subset $U$ of $\binom{[0, n]}{i + 2}$, the corresponding cubillage $\mathcal{Q}_{U}$ of $Z([0, n], i + 1)$ is determined once, for every element of $\binom{[0, n]}{i + 1}$, one knows the initial vertex of the cube with that set of generating vectors.
Hence, we write $\init{U}{L}$ for the initial vertex of the cube with generating vectors $L$ in $\mathcal{Q}_{U}$.

\begin{proposition}[{\cite[Thm.~2.1]{t02}}]
	\label{prop:init_vert}
Given a set of generating vectors~$L \in \binom{[0, n]}{i + 1}$ and $a \in [0, n]\setminus L$, we have that $a \in \init{U}{L}$ if and only if either
\begin{itemize}
\item $L \cup \{a\} \in U$ and $a$ is an even gap in $L$, or
\item $L \cup \{a\} \notin U$ and $a$ is an odd gap in $L$.
\end{itemize}
\end{proposition}

An analogous statement was shown for more general zonotopes in \cite[Lem.~5.13]{gpw22}.
It will also be useful to introduce the notation \[\initb{U}{L} := [0, n] \setminus (L \cup \init{U}{L})\] for the vectors which are neither generating vectors nor present in the initial vertex.


\section{Coproducts from cubillages}\label{sect:main}

In this section, we show how one can construct a coproduct $\Delta_{i}^{U} \colon \chains_{\bullet}(\simp^n) \to \chains_{\bullet}(\simp^n) \otimes \chains_{\bullet}(\simp^n)$ from an element $U \in \bruhat{[0, n]}{i + 1}$, or equivalently, from any cubillage $\mathcal{Q}_{U}$ of $Z([0, n], i + 1)$.
We show that all these coproducts give homotopies between $\Delta_{i - 1}$ and $T\Delta_{i - 1}$. 
The cubillage perspective allows us to give a clean and illuminating proof of this important fact (\cref{thm:homotopy_formula}).
It also allows us to show that all coproducts which satisfy the homotopy formula arise from cubillages, provided they contain no redundant terms (\cref{thm:all_coproducts}).

Basis elements of $\chains_{\bullet}(\simp^n) \otimes \chains_{\bullet}(\simp^n)$ are of the form $X \otimes Y$ for $X, Y \subseteq [0,n]$.
A central observation is as follows.
Given a basis element $X \otimes Y \in \chains_{\bullet}(\simp^n) \otimes \chains_{\bullet}(\simp^n)$, we can always write $X = L \cup A$ and $Y = L \cup B$, where $L = X \cap Y$, $A = X \setminus Y$, and $B = Y \setminus X$.
Given a subset $S \subseteq [0, n]$, we say that $L \cup A \otimes L \cup B$ is \defn{supported on~$S$} if~$L \cup A \cup B = S$.
The basic relationship between cubillages and elements of $\chains_{\bullet}(\simp^n) \otimes \chains_{\bullet}(\simp^n)$ is as follows.

\begin{proposition}\label{prop:face_term_bij}
There is a bijection between faces of $Z(S, |S|)$ excluding $\emptyset$ and~$S$ and basis elements of $\chains_{\bullet}(\simp^n) \otimes \chains_{\bullet}(\simp^n)$ which are supported on $S$.
\end{proposition}
\begin{proof}
As in Section~\ref{sect:back:hbo:cub}, we have that every face of $Z(S, \sz{S})$ is determined by its set of generating vectors $L$ and initial vertex $A$.
Defining $B := S \setminus (L\cup A)$, the corresponding basis element of $\chains_{\bullet}(\simp^n) \otimes \chains_{\bullet}(\simp^n)$ is $L \cup A \otimes L \cup B$.
This gives a well-defined element unless $L = \emptyset$ and either $A = \emptyset$ or $B = \emptyset$, which is the case if and only if the face is either of the vertices $\emptyset$ or $S$.
Bijectivity is evident.
\end{proof}

Hence, we may identify basis elements of $\chains_{\bullet}(\simp^n) \otimes \chains_{\bullet}(\simp^n)$ with the corresponding faces of $Z(S, \sz{S})$, in particular in the case $S = [0, n]$.

\begin{construction}
\label{const}
For any $U \in \bruhat{[0, n]}{i + 1}$, where $n \geqslant i$, we define the cup-$i$ coproduct 
\[\Delta_{i}^{U} \colon \chains_{\bullet}(\simp^n) \to \chains_{\bullet}(\simp^n) \otimes \chains_{\bullet}(\simp^n) \]
on the top face of $\simp^n$ by the formula 
\[\Delta_{i}^{U}([0, n]) := \sum_{L \in \binom{[0, n]}{i + 1}} (-1)^{\trmsgn(L \cup \init{U}{L} \otimes L \cup \initb{U}{L})} L \cup A_{L}^U \otimes L \cup B_{L}^U \ ,\]
where
\[\trmsgn(L \cup \init{U}{L} \otimes L \cup \initb{U}{L}) := \sum_{b \in \initb{U}{L}} \sz{\init{U}{L}}_{<b} + \sum_{l \in L} \sz{L}_{<l} + (n + 1)\sz{ \init{U}{L}} \in \mathbb{Z}/2\mathbb{Z}\ .\]
Here $\init{U}{L}$ and $\initb{U}{L}$ are the sets from Section~\ref{sect:cons->cub}.
We define this as a class in $\mathbb{Z}/2\mathbb{Z}$, since this is all that the sign depends on, and doing so makes it easier to write down the calculations in \cref{sect:appendix}.

For codimension one faces, we define \[\Delta_{i}^{U}([0, n] \setminus \{p\}) := \Delta_{i}^{U/p}([0, n] \setminus \{p\}) \ .\]
In this way, we inductively extend the definition to lower-dimensional faces too.
Once we reach a non-empty subset $S \subseteq [0, n]$ with $\sz{S} \leqslant i$, we define $\Delta_{i}^{U}(S) := 0$.
We have thus defined the map $\Delta_{i}^{U}$ for basis elements of $\chains_{\bullet}(\simp^n)$; one can then extend linearly.
\end{construction}

Hence, the terms of $\Delta_{i}^U([0, n])$ are simply those terms corresponding to the cubes of $\mathcal{Q}_{U}$ under the bijection in \cref{prop:face_term_bij}, with a certain sign attached.
In what follows, we will need to make calculations involving the signs $\trmsgn(L \cup \init{U}{L} \otimes L \cup \initb{U}{L})$.
We carry out these calculations in \cref{sect:appendix}, and refer to the relevant lemmas from there when necessary.


\subsection{Comparison to original Steenrod operations}

Now, we claim that, up to sign, for $i$ even, the original Steenrod cup-$i$ coproduct $\Delta_i$ is exactly the coproduct $\Delta_i^{\emptyset}$ coming from the minimal element $\emptyset$ the higher Bruhat order, where nothing is inverted, and the opposite of the Steenrod cup-$i$ product $T \Delta_i$ exactly comes from the maximal element of the higher Bruhat order $\binom{[0, n]}{i + 2}$, where everything is inverted. 
For $i$ odd, the opposite is true. 
The coproducts coming from other elements of the higher Bruhat orders can be thought of as intermediate coproducts between these two cases.

In order to show this, we first prove a useful proposition describing what happens to $\Delta_{i}^U$ under taking the complement of $U$.

\begin{proposition}
	\label{prop:complement}
If $U \in \bruhat{[0,n]}{i + 1}$, then we have \[\Delta_{i}^{\binom{[0,n]}{i + 2}\setminus U} = (-1)^{i}T\Delta_{i}^U \ .\]
\end{proposition}
\begin{proof}
It follows from Proposition~\ref{prop:init_vert} that
\begin{align*}
A_{L}^{\binom{[0, n]}{i + 2} \setminus U} = B_{L}^U \quad \text{ and } \quad B_{L}^{\binom{[0, n]}{i + 2} \setminus U} = A_{L}^U \ .
\end{align*}
Hence, ignoring signs, we have that $\Delta_{i}^{\binom{[0,n]}{i + 1}\setminus U}$ and $T\Delta_{i}^U$ have the same terms.
Comparing signs with Lemma~\ref{lem:sign_swap}, we have that on the left-hand side the term~$L \cup B_{L}^U \otimes L \cup A_{L}^U$ has the sign $\trmsgn(L \cup B_{L}^U \otimes L \cup A_{L}^U) = $ \[
\trmsgn(L \cup A_{L}^U \otimes L \cup B_{L}^U) + (\sz{L \cup A_{L}^U} + 1)(\sz{L \cup B_{L}^U} + 1) + \sz{L} + 1 \ ,
\]
which is precisely the sign on the right-hand side, recalling the definition of $T$ and noting that $\sz{L} = i + 1$.
\end{proof}

It is also useful to consider the following well-known fact about the initial vertices of the cubes of the lower cubillage.

\begin{lemma}\label{lem:lower_cub_spec}
For the lower cubillage of $Z([0, n], i + 1)$, which is given by $U = \emptyset \in \bruhat{[0, n]}{i + 1}$, we have
\begin{align*}
L \cup A_{L}^\emptyset &= \dots \cup [l_{i - 3}, l_{i - 2}] \cup [l_{i - 1}, l_{i}] \ , \\
L \cup B_{L}^\emptyset &= \dots \cup [l_{i - 2}, l_{i - 1}] \cup [l_{i},n] \ ,
\end{align*}
for the cube with generating vectors $L = \{l_{0}, l_{1}, \dots, l_{i}\} \in \binom{[0,n]}{i + 1}$.
\end{lemma}
\begin{proof}
This follows from \cref{prop:init_vert}, since elements of $(l_{i}, n]$ are even gaps in~$L$, elements of $(l_{i - 1}, l_{i})$ are odd gaps, and so on.
\end{proof}

We can now compare our operations to the original Steenrod coproducts.
The fact that the Steenrod coproducts alternate between corresponding to the minimal and maximal elements of the higher Bruhat orders causes a discrepancy in signs.
Our coproducts always give homotopies from the maximal element to the minimal element, so getting a homotopy in the other direction requires a minus sign.

\begin{theorem}
\label{lem:boundary=steenrod}
For $i$ even, we have \[\Delta_{i}^{\emptyset} = (-1)^{i/2}\Delta_{i} \quad  \text{ and } \quad  \Delta_{i}^{\binom{[0,n]}{i + 2}} = (-1)^{i/2} T \Delta_{i} \ ,\] whilst for $i$ odd we have \[\Delta_{i}^{\emptyset} = (-1)^{\ceil{i/2}}T\Delta_{i} \quad \text{ and } \quad \Delta_{i}^{\binom{[0,n]}{i + 2}} = (-1)^{\floor{i/2}}\Delta_{i}\ .\]
\end{theorem}
\begin{proof}
It suffices to prove the claim for $\Delta_{i}^{\emptyset}$ for $i$ even and for $\Delta_{i}^{\binom{[0,n]}{i + 2}}$ for $i$ odd, since the remaining statements are then obtained by applying Proposition~\ref{prop:complement}.
For $L = \{l_{0}, l_{1}, \dots, l_{i}\} \in \binom{[0,n]}{i + 1}$, one obtains the corresponding term of $\Delta_{i}^\emptyset$ by \cref{const} and \cref{lem:lower_cub_spec}.
If $i$ is even, then the first interval in $L \cup A_{L}^\emptyset$ is $[0,l_{0}]$, and so this is the term associated to the overlapping partition $[0, l_{1}], [l_{1}, l_{2}], \dots, [l_{i}, n]$ in the Steenrod construction.
If $i$ is odd, then the first interval in $L \cup A_{L}^\emptyset$ is $[l_{0},l_{1}]$, and so the Steenrod construction gives $L \cup A_{L}^{\binom{[0,n]}{i + 2}} \otimes L \cup B_{L}^{\binom{[0,n]}{i + 2}}$.
Since this gives a bijection between overlapping partitions of $[0, n]$ into $i + 2$ intervals and elements of $\binom{[0, n]}{i + 1}$, modulo signs the terms of $\Delta_{i}$ coincide with those of $\Delta_{i}^{\emptyset}$ for $i$ even and $\Delta_{i}^{\binom{[0,n]}{i + 2}}$ for $i$ odd.
The fact that the signs coincide follows from \cref{lem:sign:steenrod}.
\end{proof}


\subsection{The homotopy formula}

We now show how our construction can be used to give a clean and conceptual proof of the fundamental fact that $\Delta_{i}$ gives a chain homotopy between $\Delta_{i - 1}$ and~$T\Delta_{i - 1}$.

The boundary of a term in the coproduct has the following neat description.
Recalling \cref{prop:face_term_bij}, we may talk of upper and lower facets of a basis element~$F$ of $\chains_{\bullet}(\simp^n) \otimes \chains_{\bullet}(\simp^n)$, meaning the respective terms corresponding to the upper and lower facets of the zonotope face corresponding to $F$.

\begin{proposition}
	\label{prop:key}
	Let $F = L \cup A \otimes L \cup B$ be a basis element of $\chains_{\bullet}(\simp^n) \otimes \chains_{\bullet}(\simp^n)$, where $\sz{L} = i + 1$.
Adopting the notation $F/k := L \cup (A \setminus k) \otimes L \cup (B \setminus k)$ for $k \in A \cup B$, we have that
\begin{align*}
\partial ((-1)^{\trmsgn(F)} F) 
\quad = \sum_{G \text{ lower facet}}(-1)^{\trmsgn(G)}G
\quad + \sum_{H \text{ upper facet}}(-1)^{\trmsgn(H) + 1}H \\
\quad \quad \quad + \sum_{k \in A \cup B}(-1)^{\trmsgn(F/k) + k + i}F/k \ .
\end{align*}
\end{proposition}
\begin{proof}
There are four different types of terms in $\partial((-1)^{\trmsgn(F)}F)$ to consider, corresponding to Lemmas~\ref{lem:sign:L_first}, \ref{lem:sign:L_second}, \ref{lem:sign:A}, and~\ref{lem:sign:B}.
\begin{enumerate}[wide]
\item We first consider terms of $\partial((-1)^{\trmsgn(F)}F)$ given by $(L\setminus k) \cup A \otimes L \cup B$ for $k \in L$.
In the expansion of $\partial((-1)^{\trmsgn(F)}F)$, this has sign
$\trmsgn(F) + \sz{L \cup A}_{<k}$, which equals $\trmsgn((L\setminus k) \cup A \otimes L \cup B) + \sz{L}_{>k}$ by Lemma~\ref{lem:sign:L_first}.
This therefore equals $\trmsgn((L\setminus k) \cup A \otimes L \cup B)$ if and only if $k$ is an even gap in $L$.
Then by Proposition~\ref{prop:up_low_facets}, $(L\setminus k) \cup A \otimes L \cup B$ is a lower facet of $F$ if and only if $k$ is an even gap in $L$.
Upper facets then have the opposite sign to $\trmsgn((L\setminus k) \cup A \otimes L \cup B)$.

\item We now consider terms of $\partial((-1)^{\trmsgn(F)}F)$ given by $L \cup A \otimes (L \setminus k) \cup B$ for $k \in L$.
In the expansion of $\partial((-1)^{\trmsgn(F)}F)$, this has sign $\trmsgn(F) + \sz{L \cup B}_{<k} + \sz{L \cup A} + 1$, which equals
$\trmsgn(L \cup A \otimes (L \setminus k) \cup B) + \sz{L}_{>k} + 1$ by Lemma~\ref{lem:sign:L_second}.
This therefore equals $\trmsgn(L \cup A \otimes (L\setminus k) \cup B)$ if and only if $k$ is an odd gap in $L$.
Similarly to (1), by Proposition~\ref{prop:up_low_facets}, $L \cup A \otimes (L \setminus k) \cup B$ is a lower facet of $F$ if and only if $k$ is an odd gap in $L$.
Upper facets then have the opposite sign to $\trmsgn(L \cup A \otimes (L \setminus k) \cup B)$, as before.

\item We now start considering terms which do not correspond to facets of~$F$, by looking at terms given by $L \cup (A \setminus k) \otimes L \cup B$ for $k \in A$.
In the expansion of $\partial((-1)^{\trmsgn(F)}F)$, this has sign $\trmsgn(F) + (L \cup A)_{<k}$, which equals $\trmsgn(F/k) + k + \sz{L} + 1 = \trmsgn(F/k) + k + i$, by Lemma~\ref{lem:sign:A}, as desired.

\item Finally, we consider terms given by $L \cup A \otimes L \cup (B \setminus k)$ for $k \in B$.
In the expansion of $\partial((-1)^{\trmsgn(F)}F)$, this has sign
$\trmsgn(F) + \sz{L \cup B}_{<k} + \sz{L \cup A} + 1$, which equals
$\trmsgn(F/k) + k + \sz{L} + 1 = \trmsgn(F/k) + k + i$, by Lemma~\ref{lem:sign:B}.
\end{enumerate}
\end{proof}

Showing that the coproduct $\Delta_{i}^{U}$ satisfies the homotopy formula is now straightforward.

\begin{theorem}
	\label{thm:homotopy_formula}
For any $i \geqslant 0$, and for any $U \in \bruhat{[0, n]}{i + 1}$, we have that
\[\partial \circ \Delta_i^U - (-1)^{i} \Delta_i^U \circ \partial = (1 + (-1)^{i} T)\Delta_{i - 1}^\emptyset \ . \]
\end{theorem}
\begin{proof}
We reason in terms of the cubillage $\mathcal{Q}_{U}$ and its induced subdivision of~$Z([0, n], i + 1)$.
We consider the terms of $\partial \circ \Delta_{i}^{U}([0, n])$ and apply \cref{prop:key}.
The argument for other basis elements of $\chains_{\bullet}(\simp^n)$ is similar.
For $k \in [0, n]$, the terms of $\partial \circ \Delta_{i}^U([0, n])$ that correspond to $F/k$ for cubes $F$ of $\mathcal{Q}_{U}$ have sign $\trmsgn(F/k) + k + i$, by \cref{prop:key}, whereas in $-(-1)^{i}\Delta_{i}^{U}((-1)^{k}[0, n]\setminus k)$ they have sign $\trmsgn(F/k) + k + i + 1$, and so they cancel.

By \cref{prop:key}, we have that the other terms of $\partial \circ \Delta_{i}^U$ are given by facets of cubes of the cubillage.
These come in two sorts: \emph{internal facets}, which are shared between two cubes of the cubillage and lie in the interior of $Z([0, n], i + 1)$ in the induced subdivision; and \emph{boundary facets}, which are only facets of a single cube of the cubillage and lie on the boundary of $Z([0, n], i + 1)$ in the induced subdivision.

Those that correspond to internal facets of the cubillage cancel out by Proposition~\ref{prop:key}, since they are an upper facet of one term and a lower facet of another.
Hence, we are left with the terms corresponding to boundary facets.
By Proposition~\ref{prop:key} and Proposition~\ref{prop:complement}, we have that the terms corresponding to lower facets of the zonotope give $\Delta_{i - 1}^{\emptyset}$, whereas the terms corresponding to upper facets give \[- \Delta_{i - 1}^{\binom{[0,n]}{i + 2}} = -(-1)^{i - 1} T\Delta_{i - 1}^{\emptyset},\] as desired.
\end{proof}

Note that in proving Theorem~\ref{thm:homotopy_formula}, all of the work went into proving that the signs matched up as desired.
Working with chain complexes of $\ZZ$-modules rather than $\mathbb{Z}$-modules, the result is immediate from the cubillage perspective.
We will work with $\ZZ$-modules when we consider Steenrod squares in \cref{sect:squares}.

\begin{figure}[!htb]
		\centerline{
		\resizebox{0.75\linewidth}{!}{
		\begin{tikzpicture}[scale=2.4,xscale=1.6]
		
		
		\coordinate(e) at (0,0);
		\coordinate(0) at (-1,1.4);
		\coordinate(01) at (-1,2.8);
		\coordinate(012) at (0,4.2);
		\coordinate(12) at (1,2.8);
		\coordinate(2) at (1,1.4);
		
		\draw (e) -- (0) -- (01) -- (012) -- (12) -- (2) -- (e);
		
		
		\coordinate(02) at (0,2.8);
		
		\draw (012) -- (02);
		\draw (02) -- (0);
		\draw (02) -- (2);
		
		
		\node at (e) [below = 1mm of e] {\tiny $\emptyset$};
		\node at (0) [left = 1mm of 0] {\tiny $0$};
		\node at (01) [left = 1mm of 01] {\tiny $01$};
		\node at (012) [above = 1mm of 012] {\tiny $012$};
		\node at (2) [right = 1mm of 2] {\tiny $2$};
		\node at (12) [right = 1mm of 12] {\tiny $12$};
		
		\node at (02) [below = 1mm of 02] {\tiny $02$};
		
		
		\node at ($(2)!0.5!(012)$) {\color{red} $+012 \otimes 01$};
		\node at ($(e)!0.5!(02)$) {\color{red} $-02 \otimes 012$};
		\node at ($(0)!0.5!(012)$) {\color{red} $+012 \otimes 12$};
		
		
		\coordinate(0012) at ($(e)!0.5!(0)$);
		\coordinate(0112) at ($(0)!0.5!(01)$);
		\coordinate(0122) at ($(01)!0.5!(012)$);
		
		\coordinate(0121) at ($(02)!0.5!(012)$);
		\coordinate(0201) at ($(2)!0.5!(02)$);
		\coordinate(0212) at ($(0)!0.5!(02)$);
		
		\coordinate(2012) at ($(e)!0.5!(2)$);
		\coordinate(1201) at ($(2)!0.5!(12)$);
		\coordinate(0120) at ($(12)!0.5!(012)$);
		
		
		\node at (0012) [left = 1mm of 0012] {\color{blue} \tiny $+0 \otimes 012$};
		\node at (0112) [left = 1mm of 0112] {\color{blue} \tiny $+01 \otimes 12$};
		\node at (0122) [left = 1mm of 0122] {\color{blue} \tiny $+012 \otimes 2$};
		
		\node at (0121) [left = 0.1mm of 0121] {\color{OliveGreen} \tiny $-012 \otimes 1$};
		\node at (0121) [right = 0.1mm of 0121] {\color{OliveGreen} \tiny $+012 \otimes 1$};
		
		\node [rotate=43] at (0212) [below = 1mm of 0212] {\color{OliveGreen} \tiny $+02 \otimes 12$};
		\node [rotate=43] at (0212) [above = 1mm of 0212] {\color{OliveGreen} \tiny $-02 \otimes 12$};
		
		\node [rotate=-43] at (0201) [above = 1mm of 0201] {\color{OliveGreen} \tiny $-02 \otimes 01$};
		\node [rotate=-43] at (0201) [below = 1mm of 0201] {\color{OliveGreen} \tiny $+02 \otimes 01$};
		
		\node at (2012) [right = 1mm of 2012] {\color{blue} \tiny $-2 \otimes 012$};
		\node at (1201) [right = 1mm of 1201] {\color{blue} \tiny $+12 \otimes 01$};
		\node at (0120) [right = 1mm of 0120] {\color{blue} \tiny $-012 \otimes 0$};
		
		\end{tikzpicture}}}
		\quad \quad
		\centerline{
		\resizebox{0.75\linewidth}{!}{
		\begin{tikzpicture}[scale=2.4,xscale=1.6]
		
		
		\coordinate(e) at (0,0);
		\coordinate(0) at (-1,1.4);
		\coordinate(01) at (-1,2.8);
		\coordinate(012) at (0,4.2);
		\coordinate(12) at (1,2.8);
		\coordinate(2) at (1,1.4);
		
		\draw (e) -- (0) -- (01) -- (012) -- (12) -- (2) -- (e);
		
		
		\coordinate(1) at (0,1.4);
		
		\draw (e) -- (1);
		\draw (1) -- (01);
		\draw (1) -- (12);
		
		
		\node at (e) [below = 1mm of e] {\tiny $\emptyset$};
		\node at (0) [left = 1mm of 0] {\tiny $0$};
		\node at (01) [left = 1mm of 01] {\tiny $01$};
		\node at (012) [above = 1mm of 012] {\tiny $012$};
		\node at (2) [right = 1mm of 2] {\tiny $2$};
		\node at (12) [right = 1mm of 12] {\tiny $12$};
		
		\node at (1) [above = 1mm of 1] {\tiny $1$};
		
		
		\node at ($(e)!0.5!(01)$) {\color{red} $-01 \otimes 012$};
		\node at ($(e)!0.5!(12)$) {\color{red} $-12 \otimes 012$};
		\node at ($(1)!0.5!(012)$) {\color{red} $+012 \otimes 02$};
		
		
		\coordinate(0012) at ($(e)!0.5!(0)$);
		\coordinate(0112) at ($(0)!0.5!(01)$);
		\coordinate(0122) at ($(01)!0.5!(012)$);
		
		\coordinate(1012) at ($(e)!0.5!(1)$);
		\coordinate(0102) at ($(1)!0.5!(01)$);
		\coordinate(1202) at ($(1)!0.5!(12)$);
		
		\coordinate(2012) at ($(e)!0.5!(2)$);
		\coordinate(1201) at ($(2)!0.5!(12)$);
		\coordinate(0120) at ($(12)!0.5!(012)$);
		
		
		\node at (0012) [below left = 1mm of 0012] {\color{blue} \tiny $+0 \otimes 012$};
		\node at (0112) [left = 1mm of 0112] {\color{blue} \tiny $+01 \otimes 12$};
		\node at (0122) [left = 1mm of 0122] {\color{blue} \tiny $+012 \otimes 2$};
		
		\node at (1012) [above left = 0.1mm of 1012] {\color{OliveGreen} \tiny $-1 \otimes 012$};
		\node at (1012) [above right = 0.1mm of 1012] {\color{OliveGreen} \tiny $+1 \otimes 012$};
		
		\node [rotate=-43] at (0102) [below = 1mm of 0102] {\color{OliveGreen} \tiny $-01 \otimes 02$};
		\node [rotate=-43] at (0102) [above = 1mm of 0102] {\color{OliveGreen} \tiny $+01 \otimes 02$};
		
		\node [rotate=43] at (1202) [above = 1mm of 1202] {\color{OliveGreen} \tiny $+12 \otimes 02$};
		\node [rotate=43] at (1202) [below = 1mm of 1202] {\color{OliveGreen} \tiny $-12 \otimes 02$};
		
		\node at (2012) [right = 1mm of 2012] {\color{blue} \tiny $-2 \otimes 012$};
		\node at (1201) [right = 1mm of 1201] {\color{blue} \tiny $+12 \otimes 01$};
		\node at (0120) [right = 1mm of 0120] {\color{blue} \tiny $-012 \otimes 0$};
		
		\end{tikzpicture}}}
		\caption{The upper and lower cubillages of $Z(3, 2)$, the associated $\Delta_{1}$ coproducts, and their boundaries}
		\label{fig:z32_lower_upper}
\end{figure}
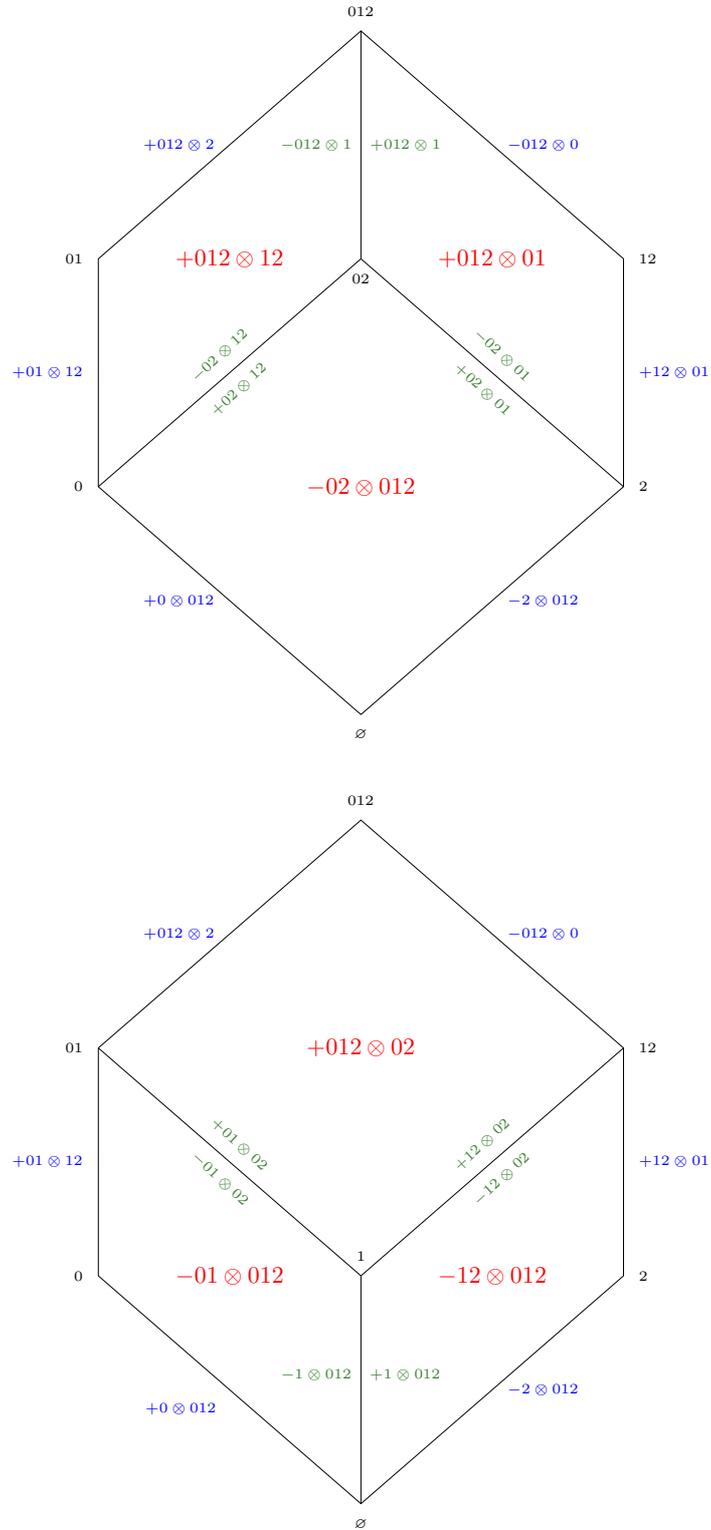

\begin{example}\label{ex:main}
	We illustrate how cubillages can be used to deduce the homotopy formula for cup-$i$ coproducts. 
	For the cup-$1$ case and $n = 2$, corresponding to the higher Bruhat poset $\bruhat{[0,2]}{2}$, there are two possible cubillages, giving two cup-$1$ coproducts. 
	These correspond to the Steenrod coproduct $\Delta_{1}$ and its opposite $T\Delta_{1}$, modulo signs.
	The cup-$0$ coproduct is 
	\[\Delta_{0}(012) = 0 \otimes 012 + 01 \otimes 12 + 012 \otimes 2 \ .\] 
	The opposite coproduct is given by 
	\[T\Delta_{0}(012) = 012 \otimes 0 - 12 \otimes 01 + 2 \otimes 012 \ .\]
	Note that $\Delta_{0} = \Delta_{0}^{\emptyset}$.
	For the two possible cup-1 coproducts $\Delta_{1}^{U}$, we wish to show that \[\Delta_{1}^{U} \circ \partial + \partial \circ \Delta_{1}^{U} = \Delta_{0} - T\Delta_{0} \ ,\] that is, the homotopy formula holds. 
	By \cref{const} and \cref{fig:z32_lower_upper}, our two cup-1 coproducts are
	\begin{align*}
	\Delta_{1}^{\{012\}}(012) &= 012 \otimes 01 - 02 \otimes 012 + 012 \otimes 12 \ , \\
	\Delta_{1}^{\emptyset}(012) &= - 01 \otimes 012 + 012 \otimes 02 - 12 \otimes 012 \ .
	\end{align*}
	Here the first is the Steenrod coproduct $\Delta_{1}$ and the second is $-T\Delta_{1}$; compare \cref{ex:back:steenrod}. 
	We have that the first coproduct comes from the cubillage at the top of \cref{fig:z32_lower_upper}, whereas the second comes from the bottom cubillage.
	We now verify the homotopy formula for $\Delta_{1}^{\emptyset}$.
	On the one hand, we have
	\begin{align*}
	\Delta_{1}^{\emptyset} \circ \partial(012) &= \Delta_{1}^{\emptyset}(12) - \Delta_{1}^{\emptyset}(02) + \Delta_{1}^{\emptyset}(01) =  - 12 \otimes 12 + 02 \otimes 02 - 01 \otimes 01 \ , 
	\end{align*}
	while on the other hand
	\begin{align*}
	\partial \circ \Delta_{1}^{\emptyset}(012) &= - \partial(01 \otimes 012) + \partial(012 \otimes 02) - \partial(12 \otimes 012) \\
	&= - 1 \otimes 012 + 0 \otimes 012 + 01 \otimes 12 - 01 \otimes 02 + 01 \otimes 01 \\
	&\quad + 12 \otimes 02 - 02 \otimes 02 + 01 \otimes 02 + 012 \otimes 2 - 012 \otimes 0 \\
	&\quad - 2 \otimes 012 + 1 \otimes 012 + 12 \otimes 12 - 12 \otimes 02 + 12 \otimes 01 \\
	&= (0 \otimes 012 + 01 \otimes 12 + 012 \otimes 2) - (2 \otimes 012 - 12 \otimes 01 + 012 \otimes 0) \\
	&\quad + (\cancel{1 \otimes 012} - \cancel{1 \otimes 012}) + (\bcancel{01 \otimes 02} - \bcancel{01 \otimes 02}) + (\cancel{12 \otimes 02} - \cancel{12 \otimes 02})  \\
	&\quad + 12 \otimes 12 - 02 \otimes 02 + 01 \otimes 01 \ .
	\end{align*} 
	Note that the terms from $\Delta^{\emptyset}_{1} \circ \partial$ cancel out, as can be seen in the final line above, and that the remaining terms from $\partial \circ \Delta^{\emptyset}_{1}$ come from faces of the cubes of the cubillage.
	Internal faces have terms coming from two cubes, which cancel each other out, as seen in the penultimate line.
	After all cancellations, the remaining terms come from the boundary of the zonotope, which give precisely $\Delta_{0} - T\Delta_{0}$, as seen in the third last line.
	The top of Figure~\ref{fig:z32_lower_upper} indicates how the verification of the homotopy formula works for $\Delta_{1}^{\{012\}}$. 
\end{example}

\begin{example}
	The figure appearing on the first page of the paper is an example of a coproduct for the $4$-simplex which does not come from the Steenrod coproduct or its opposite.
	We do not carry out the full verification of the homotopy formula, but instead illustrate in the figure that terms corresponding to internal faces of the cubillage cancel.
\end{example}

In fact, Construction~\ref{const} comprises \emph{all} coproducts that satisfy the homotopy formula, up to redundancies.
The idea is to run the proof of Theorem~\ref{thm:homotopy_formula} in reverse, so that if a coproduct satisfies the homotopy formula, then the cubes corresponding to its terms must have come from a cubillage.

\begin{theorem}\label{thm:all_coproducts}
Suppose that we have a coproduct $\Delta'_{i}\colon \chains_{\bullet}(\simp^n) \to \chains_{\bullet}(\simp^n) \otimes \chains_{\bullet}(\simp^n)$ with $i \geqslant 0$ such that
\begin{equation}\label{eq:all_coproducts}
\partial \circ \Delta'_i - (-1)^{i}\Delta'_i \circ \partial = (1 + (-1)^{i} T)\Delta_{i - 1}^\emptyset \ .
\end{equation}
If $i = 0$ suppose further that $\Delta'_{i}(p) = p \otimes p$ for all $p \in [0, n]$, and suppose otherwise that for all non-empty $S \subseteq [0, n]$, $\Delta'_{i}(S)$ has a minimal number of terms amongst coproducts which satisfy \eqref{eq:all_coproducts}.
Then we have $\Delta'_{i} = \Delta_{i}^U$ for some $U \in \bruhat{[0,n]}{i + 1}$.
\end{theorem}
\begin{proof}
We prove the result by showing that for all non-empty $S \subseteq [0, n]$, there exists $U \in \bruhat{S}{i + 1}$ such that for all non-empty $S' \subseteq S$, $\Delta'_{i}(S') = \Delta_{i}^{U}(S')$.
We use induction on the size of $S$; the result is then given by the case $S = [0, n]$.
For $\sz{S} \leqslant i - 1$, we have $\Delta_{i - 1}^{\emptyset}(S) = 0$, so the right-hand side of~\eqref{eq:all_coproducts} is zero.
For $\sz{S} = i$, we have $\Delta_{i - 1}^{\emptyset}(S) = \pm S \otimes S$, and so $T\Delta_{i - 1}^{\emptyset}(S) = (-1)^{i - 1}\Delta_{i - 1}^{\emptyset}(S)$ and the right-hand side of~\eqref{eq:all_coproducts} is also zero.
Hence, for $\sz{S} \leqslant i$, $\Delta'_{i}(S) = 0$ satisfies~\eqref{eq:all_coproducts} and has the minimal number of terms amongst such coproducts.

We now consider the case $\sz{S} = i + 1$, which is the first non-trivial base case.
We have that $(\Delta'_{i} \circ \partial)(S) = 0$ by the previous paragraph.
For $i > 0$, the right-hand side of \eqref{eq:all_coproducts} has terms corresponding to the lower facets of $Z(S, i + 1)$ minus the upper facets, and so by \cref{prop:key}, we have that $\Delta'_{i}(S) = \pm S \otimes S$ satisfies~\eqref{eq:all_coproducts}.
Note that the sign is determined by the right-hand side of \eqref{eq:all_coproducts} and that \cref{prop:key} also shows that no other single term will satisfy~\eqref{eq:all_coproducts}.
For $i = 0$, we have that $\Delta'_{i}(S) = S \otimes S$ by assumption.
These base cases then hold, since we have $\Delta'_{i}(S) = \Delta_{i}^{U}(S)$ for the unique element $U \in \bruhat{S}{i + 1}$; for proper subsets $S'$ of $S$, we must have $\Delta'_{i}(S') = \Delta_{i}^{U}(S') = 0$ by the previous paragraph.

We now show the inductive step, where we suppose that $\sz{S} = k > i + 1$.
We know from the induction hypothesis that there exist $U_{s_p} \in \bruhat{S \setminus s_{p}}{i + 1}$ such that $\Delta'_{i}(S') = \Delta_{i}^{U_{s_p}}(S')$ for all non-empty $S' \subseteq S \setminus s_{p}$, where $S = \{s_{1}, s_{2}, \dots, s_{\sz{S}}\}$.
For each $L \in \binom{S}{i + 1}$, let $s_{L} = \min S \setminus L$ and let $X_{L} \otimes Y_{L}$ be the term of $\Delta'_{i}(S)$ which has in its boundary the term $A_{L}^{U_{s_{L}}} \cup L \otimes B_{L}^{U_{s_{L}}} \cup L$ of $\Delta'_{i}(S \setminus s_{L})$.

Note that we may talk about generating vectors of a term analogously to those of a face by \cref{prop:face_term_bij}.
We claim first that each term $X_{L} \otimes Y_{L}$ has $i + 1$ generating vectors, namely given by~$L$.
Indeed, since $X_{L} \otimes Y_{L}$ has $A_{L}^{U_{s_{L}}} \cup L \otimes B_{L}^{U_{s_{L}}} \cup L$ in its boundary, its generating vectors must either consist of $L$ or $L \cup \{s\}$ for $s \in [0, n] \setminus L$.
Hence, suppose that we have $X_{L} = L \cup \{s\} \cup A$ and $Y_{L} = L \cup \{s\} \cup B$ for some disjoint $A, B \subseteq S \setminus (L \cup \{s, s_{L}\})$ with  $L \cup \{s\} \cup A \cup B = S \setminus s_{L}$.
However, $X_{L} \otimes Y_{L}$ then has both $L \cup A \otimes L \cup \{s\} \cup B$ and $L \cup \{s\} \cup A \otimes L \cup B$ in its boundary.
Thus, after cancelling $\partial(X_{L} \otimes Y_{L})$ with $(\Delta'_{i} \circ \partial)(S)$, there must still be a term supported on $S \setminus \{s_{L}\}$ with generating vectors $L$.
Thus terms supported on $S \setminus \{s_{L}\}$ with generating vectors $L$ can never be completely cancelled by the boundary of $X_{L'} \otimes Y_{L'}$ for any $L' \in \binom{S}{i + 1}$ if $X_{L'} \otimes Y_{L'}$  has $i + 2$ generating vectors.
We conclude that we must have $X_{L} = L \cup A$, $Y_{L} = L \cup B$, with $L$, $A$, $B$ disjoint.
Moreover, each term $X_{L} \otimes Y_{L}$ is then distinct.
There are then $\binom{\sz{S}}{i + 1}$ of these terms, which is the number of cubes of a cubillage of $Z(S, i + 1)$.
Since we have assumed that $\Delta'_{i}(S)$ has the minimum number of terms possible, we deduce that the terms $X_{L} \otimes Y_{L}$ comprise all terms of $\Delta'_{i}(S)$.

We now show that $X_{L} \otimes Y_{L}$ is actually supported on $S$.
The alternative is that it is supported on $S \cup \{q\} \setminus s_{L}$ for $q \notin S \cup \{s_{L}\}$.
If $A \cup B \neq \emptyset$, then taking the boundary at an element of $A$ or $B$ gives a term with generating vectors $L$ which is not supported on $S$ and so cannot cancel with any of the other terms in the formula~\eqref{eq:all_coproducts}.
Indeed, a consequence of the previous paragraph is that all terms of $(\partial \circ \Delta'_{i})(S)$ with generating vectors $L$ are in the boundary of $X_{L} \otimes Y_{L}$.
If $A \cup B = \emptyset$, then $\sz{S} = i + 2$ and $X_{L'} \otimes Y_{L'} = L' \cup \{q_{L'}\} \otimes L'$ or $L' \otimes L' \cup \{q_{L'}\}$ for some $q_{L'} \in [0, n] \setminus L'$ all~$L' \in \binom{S'}{i + 1}$.
If $q_{L} \notin S$, then boundaries of $X_{L} \otimes Y_{L}$ can only cancel with those of $X_{L'} \otimes Y_{L'}$ if $q_{L'} = q_{L}$.
But even in this case we cannot get cancellations, since $L$ and $L'$ differ by at least one element.
Thus $q_{L'} \in S$ for all~$L'$, as required. 

Using \cref{prop:face_term_bij} we can then identify the terms $X_{L} \otimes Y_{L}$ of $\Delta'_{i}(S)$ with $(i + 1)$-dimensional faces of $Z(S, \sz{S})$ with generating vectors~$L$.
It is clear from \cref{prop:key} that if these faces do not form a cubillage of $Z(S, i + 1)$, then terms of $(\partial \circ \Delta'_{i})(S)$ from facets of these faces will not cancel.
The signs of the terms of $\Delta'_{i}(S)$ are moreover determined by the signs of the terms from the facets of $Z(S, i + 1)$ given by the right-hand side of~\eqref{eq:all_coproducts}.
Hence, there is a cubillage $U \in \bruhat{S}{i + 1}$ such that $\Delta'_{i}(S) = \Delta_{i}^{U}(S)$.

It then follows from \cref{prop:key} that the terms of $(\partial \circ \Delta'_{i}(S))$ supported on $S \setminus s_{p}$ are given by $\Delta_{i}^{U/s_{p}}$.
If these terms are to cancel, we must have $U/s_{p} = U_{s_p}$, and so we indeed have $\Delta'_{i}(S') = \Delta_{i}^{U}(S')$ for all non-empty $S' \subseteq S$, as desired.
The result then follows by induction.
\end{proof}

\begin{remark}
Of course, one can find other coproducts which satisfy the homotopy formula by taking $\Delta'_{i} := \Delta_{i}^U - \Delta_{i}^{U'} + \Delta_{i}^{U''}$ for $U, U', U'' \in \bruhat{[0, n]}{i + 1}$, for instance.
\end{remark}


\subsection{Homotopies from covering relations}\label{sect:cov_rel}

We now show how, given $U, V \in \bruhat{[0, n]}{i}$ with $U \lessdot V$ a covering relation, one can construct a homotopy from $\Delta_{i - 1}^{V}$ to $\Delta_{i - 1}^U$.
This is essentially just a different way of recasting \cref{const}, but this perspective will be useful in \cref{sect:imp:reorient}.
We will obtain an alternative proof of Theorem~\ref{thm:homotopy_formula}: since $\Delta_{i - 1}^{\emptyset}$ and $(-1)^{i - 1}T\Delta_{i - 1}^\emptyset$ correspond to the minimal and maximal elements of $\bruhat{[0, n]}{i}$, we can take the sequence of homotopies corresponding to any maximal chain of covering relations in $\bruhat{[0, n]}{i}$.
Such a maximal chain then gives an element of~$\bruhat{[0, n]}{i + 1}$ and the coproduct giving the homotopy will be precisely the one from \cref{const}.

\begin{construction}\label{const:cov_rel}
Let $F$ be a face of $Z([0, n], n + 1)$ with generating vectors $L$ and initial vertex $A$ such that $\sz{L} = i + 1$, with $B := [0, n] \setminus (L \cup A)$ as usual.
We define a coproduct
\[\Delta_{i}^F \colon \chains_{\bullet}(\simp^n) \to \chains_{\bullet}(\simp^n) \otimes \chains_{\bullet}(\simp^n)\]
by $\Delta_{i}^F([0, n]) := (-1)^{\trmsgn(L \cup A \otimes L \cup B)} L \cup A \otimes L \cup B$.
Then we extend inductively to lower-dimensional faces analogously to the usual way:
\[\Delta_{i}^F([0, n] \setminus p) := (-1)^{\trmsgn(L \cup (A\setminus p) \otimes L \cup (B\setminus p))} L \cup (A \setminus p) \otimes L \cup (B \setminus p)\]
if $p \notin L$, and $\Delta_{i}^F([0, n] \setminus p) = 0$ if $p \in L$.
Here we have only specified the coproduct on basis elements, with the values on other elements obtained by extending linearly.
\end{construction}

We have the following ``local'' version of the homotopy formula.

\begin{theorem}\label{thm:cov_rel}
Let $U, V \in \bruhat{[0, n]}{i}$ such that $U \lessdot V$, with this covering relation given by the dimension $i + 1$ face $F$ of $Z([0, n], n + 1)$.
Then we have
\begin{equation}\label{eq:cov_rel}
\partial \circ \Delta_{i}^F - (-1)^i \Delta_{i}^F \circ \partial = \Delta_{i - 1}^U - \Delta_{i - 1}^{V} \ .
\end{equation}
That is, $\Delta_{i}^F$ gives a homotopy from $\Delta_{i - 1}^{V}$ to $\Delta_{i - 1}^{U}$.
\end{theorem}
\begin{proof}
By Construction~\ref{const}, we have that
\[\Delta_{i - 1}^{U}([0, n]) - \Delta_{i - 1}^{V}([0, n]) = \sum_{\substack{G \text{ lower}\\ \text{facet of } F}} (-1)^{\trmsgn(G)}G + \sum_{\substack{H \text{ upper} \\ \text{facet of } F}} (-1)^{\trmsgn(H) + 1}H \ , \]
since, by assumption, the cubes of $U$ and $V$ only differ in that $U$ contains the lower facets of $F$ whilst $V$ contains the upper facets.
It then follows from Proposition~\ref{prop:key} that $\partial \circ \Delta_{i}^F([0, n]) = \Delta_{i - 1}^U([0, n]) - \Delta_{i - 1}^{V}([0, n]) + (-1)^i \Delta_{i}^F \circ \partial([0, n])$.
This is the desired statement for $[0, n]$, with the result following similarly for other basis elements of $\chains_{\bullet}(\simp^n)$.
\end{proof}

We can now give an alternative proof of the homotopy formula.

\begin{proof}[Alternative proof of \cref{thm:homotopy_formula}]
If we let $U \in \bruhat{[0, n]}{i + 1}$, then, comparing Construction~\ref{const} and Construction~\ref{const:cov_rel}, we have that \[\Delta_{i}^U = \sum_{F \text{ cube of }U} \Delta_{i}^F \ .\]
We then have that \[\partial \circ \Delta_{i}^U - (-1)^i \Delta_{i}^U \circ \partial = \sum_{F \text{ cube of }U} \bigl(\partial \circ \Delta_{i}^F - (-1)^i \Delta_{i}^F \circ \partial \bigr) \ .\]
According to the fundamental theorem of the higher Bruhat orders (\cref{thm:fund_hbo}), the faces of $U$ can be ordered such that they form a maximal chain of covering relations in $\bruhat{[0, n]}{i}$.
Hence, using the above \cref{thm:cov_rel}, we have that \[\sum_{F \text{ cube of }U} \bigl( \partial \circ \Delta_{i}^F - (-1)^i \Delta_{i}^F \circ \partial \bigr) = \Delta_{i - 1}^\emptyset - \Delta_{i - 1}^{\binom{[0, n]}{i + 1}} \ ,\] 
which is equal to $\Delta_{i - 1}^\emptyset + (-1)^iT\Delta_{i - 1}^{\emptyset}$ by \cref{prop:complement}.
\end{proof}

This gives a clear conceptual picture of our construction: to any element in the higher Bruhat orders, it associates a linear coproduct; and to any relation between two elements in the higher Bruhat orders, it associates a chain homotopy between the corresponding coproducts. 


\section{Implications of the construction}
\label{sect:imp}

In this section, we extend our construction of cup-$i$ coproducts to simplicial complexes (\cref{sect:imp:simp_comp}), and also consider singular homology (\cref{sect:imp:sing}).
In this latter case, under natural restrictions, only the Steenrod coproduct and its opposite are possible.
Next, we show how to construct homotopies between $\Delta_{i}^U$ and~$T\Delta_{i}^U$ for arbitrary $U \in \bruhat{[0, n]}{i+1}$, using the so-called ``reoriented higher Bruhat orders" (\cref{sect:imp:reorient}).
We finish by showing that our coproducts can be used to define Steenrod squares in cohomology and that, in fact, they all induce the same Steenrod squares.


\subsection{Simplicial complexes}
\label{sect:imp:simp_comp}

So far we have only considered coproducts on the chain complex of the simplex, but it is natural to ask for arbitrary simplicial complexes what coproducts $\Delta'_{i}$ exist and give a homotopy between $\Delta_{i - 1}^\emptyset$ and $T\Delta_{i - 1}^\emptyset$.

A \defn{finite simplicial complex} is a set $\Sigma$ of non-empty subsets of $[0, n]$ such that for each $\sigma \in \Sigma$, and non-empty $\tau \subseteq \sigma$, we have that $\tau \in \Sigma$.
Given a finite simplicial complex $\Sigma$, one can form its \defn{geometric realisation} $||\Sigma||$ by taking a $p$-simplex for every $\sigma \in \Sigma$ with $\sz{\sigma} = p+1$, and gluing them together according to $\Sigma$, see \cite[Sec.~2.1]{h02}.
We will hence refer to the elements of $\Sigma$ as \defn{simplices}.
We write $\Sigma_{p}$ for the set of \defn{$p$-simplices} of $\Sigma$, that is, simplices $\sigma \in \Sigma$ with $\sz{\sigma} = p+1$.
A simplex $\sigma \in \Sigma$ is \defn{maximal} if there is no $\tau \in \Sigma \setminus \{\sigma\}$ such that $\tau \supset \sigma$.
We denote the set of maximal simplices of $\Sigma$ by $\max(\Sigma)$.
The cellular chain complex $\chains_{\bullet}(\Sigma)$ of $\Sigma$ has $\Sigma_{p}$ as basis of $\chains_{p}(\Sigma)$, with the boundary map defined as in \cref{sect:back:at_conv}.

Given a finite simplicial complex $\Sigma$ on $[0, n]$, our analogue of the higher Bruhat poset is as follows.

\begin{definition}
Given a set of subsets $U \subseteq \binom{[0, n]}{i + 2}$ and a subset $M \in \binom{[0, n]}{i + 3}$, we say that $U$ is \defn{consistent with respect to $M$} if and only if the intersection $\packet{M} \cap U$ is a beginning segment of $\packet{M}$ in the lexicographic order or an ending segment.

We then say that $U$ is \defn{$\Sigma$-consistent} if it is consistent with respect to all $(i + 2)$-simplices of $\Sigma$.
We write $\bruhat{\Sigma}{i + 1}$ for the set of $\Sigma$-consistent subsets of $\Sigma_{i + 1}$.
\end{definition}

As the notation suggests, we have that $\{U \cap \Sigma_{i + 1} \st U \in \bruhat{[0,n]}{i + 1}\} \subset \bruhat{\Sigma}{i + 1}$, with the equality $\bruhat{\Sigma}{i + 1}=\bruhat{[0,n]}{i + 1}$ if $\Sigma$ is a $n$-simplex.
However, the reverse inclusion $\{U \cap \Sigma_{i + 1} \st U \in \bruhat{[0,n]}{i + 1}\} \supseteq \bruhat{\Sigma}{i + 1}$ does not hold in general, as the following example shows. 

\begin{example}\label{ex:non_order_complex}
Consider the simplicial complex $\Sigma$ on $[0, 3]$ with $\max(\Sigma) = \{012, 013, 023, 123\}$.
Then $\{012, 123\}$ is a $\Sigma$-consistent subset of $\Sigma_{2}$, and so an element of $\bruhat{\Sigma}{2}$, since $\Sigma$ has no 3-simplices.
However, this is clearly not a consistent subset of $\binom{[0, 3]}{3}$ in the usual sense.
\end{example}

We observe that $\bruhat{\Sigma}{i + 1}$ is just a set, rather than a poset.
It is not clear in general whether there is a sensible partial order on $\bruhat{\Sigma}{i + 1}$.
For an element of the higher Bruhat orders $U \in \bruhat{[0, n]}{i + 1}$, if $U \neq \binom{[0, n]}{i + 2}$, there is always $K \in \binom{[0, n]}{i + 2} \setminus U$ such that $U \cup \{K\}$ is consistent.
This is not the case for $\bruhat{\Sigma}{i + 1}$, as the following example shows.

\begin{example}\label{ex:sigma_p_cycles}
Consider the simplicial complex $\Sigma$ on $[0, 6]$ with \[
\max(\Sigma) = \{0135, 0145, 0235, 0245\}
\] and consider \[
U = \{013, 024, 145, 235\}.
\]
Then $U$ is a $\Sigma$-consistent subset of~$\Sigma_{2}$.
However, if we were to try to add $025$ to $U$, then consistency with respect to $0235$ requires that we add $035$ as well.
But then consistency with respect to $0135$ requires that we also add $015$.
In turn, consistency with respect to $0145$ requires also adding $045$.
Finally, consistency with respect to $0245$ requires adding $025$, which is what we were trying to add in the first place.
Note that there are no other $2$-simplices that can be added to $U$ whilst preserving $\Sigma$-consistency.
\end{example}

Hence, one cannot hope for covering relations in $\bruhat{\Sigma}{i + 1}$ given by adding single subsets; subsets have to be added to $U$ simultaneously, rather than one-by-one, in order to preserve consistency.
It is worth also saying that having the relation on $\bruhat{\Sigma}{i + 1}$ as straightforward inclusion does not work, since this does not coincide with the relation on $\bruhat{[0, n]}{i + 1}$ \cite[Thm~4.5]{z93}.

Next, we observe that restricting an element of the poset $\bruhat{\Sigma}{i + 1}$ to a face $\sigma$ of $\Sigma$ gives an element of the corresponding poset $\bruhat{\sigma}{i + 1}$.

\begin{lemma}\label{lem:simp_comp_simplex}
If $U \in \bruhat{\Sigma}{i + 1}$ and $\sigma \in \Sigma$, then $U_{\sigma} := U \cap \binom{\sigma}{i + 2} \in \bruhat{\sigma}{i + 1}$.
\end{lemma}
\begin{proof}
It follows from the definition of a simplicial complex that $\binom{\sigma}{i + 2} \subseteq \Sigma$.
Therefore, $U_{\sigma}$ is a consistent subset of $\binom{\sigma}{i + 2}$ in the usual sense, and so $U_{\sigma} \in \bruhat{\sigma}{i + 1}$.
\end{proof}

We can now construct coproducts on the cellular chain complex $\chains_{\bullet}(\Sigma)$ of any finite simplicial complex $\Sigma$ as follows.

\begin{construction}\label{const:simp_comp}
Let $\Sigma$ be a finite simplicial complex on $[0, n]$, let $U \in \bruhat{\Sigma}{i + 1}$, and given a simplex $\sigma \in \Sigma$, let us write $U_{\sigma} := U \cap \binom{\sigma}{i + 2}$.
Then, we define
\begin{align*}
\Delta^U_{i} \colon \chains_{\bullet}(\Sigma) &\to \chains_{\bullet}(\Sigma) \otimes \chains_{\bullet}(\Sigma) \\
\sigma &\mapsto \Delta_{i}^{U_{\sigma}}(\sigma) 
\end{align*}
\end{construction}

As in the case of the standard simplex, the coproducts $\Delta_{i}^U$ on $\chains_{\bullet}(\Sigma)$ from \cref{const:simp_comp} satisfy the homotopy formula, and all reasonable families of coproducts satisfying the homotopy formula arise from this construction.

\begin{theorem} \label{thm:simp_comp}
For any $U \in \bruhat{\Sigma}{i + 1}$ we have that
\[\partial \circ \Delta^U_{i} - (-1)^i \Delta^U_{i} \circ \partial = (1 + (-1)^i T)\Delta_{i - 1}^\emptyset \ .\]
Moreover, any coproduct $\Delta'_{i}$ which satisfies this formula, for which $\Delta'_{i}(\sigma) = \sigma \otimes \sigma$ for all $\sigma \in \Sigma_{0}$ if $i = 0$, and which otherwise has a minimal number of terms, arises from \cref{const:simp_comp}.
\end{theorem}
\begin{proof}
The homotopy formula follows from applying \cref{thm:homotopy_formula} simplex by simplex.
Indeed, using the notation for contraction from \cref{sect:back:hbo:cons}, we have that if $\tau \subseteq \sigma$, then $U_{\tau} = U_{\sigma}/(\sigma \setminus \tau)$.
Hence \cref{const:simp_comp} ensures that the terms from $(-1)^i \Delta^U_{i} \circ \partial$ cancel the terms from $\partial \circ \Delta^U_{i}$ correctly.

To see that any coproduct $\Delta'_{i}$ with a minimal number of terms arises in this way, we apply \cref{thm:all_coproducts} to the individual simplices of $\Sigma$.
Indeed, for a simplex $\sigma \in \Sigma$, by \cref{thm:all_coproducts}, we must have that there exists some some $U(\sigma) \in \bruhat{\sigma}{i + 1}$ such that $\Delta'_{i}(\tau) = \Delta_{i}^{U(\sigma)}(\tau)$ for all $\tau \subseteq \sigma$.
We then define $U := \bigcup_{\sigma \in \Sigma} U(\sigma)$.
We claim that $U(\sigma) = U_{\sigma} := U \cap \binom{\sigma}{i + 2}$.
Indeed, for any $\sigma, \sigma' \in \Sigma$ with $\sigma \cap \sigma' \neq \emptyset$, we have that $\Delta^{U(\sigma)}(\tau) = \Delta'_{i}(\tau) = \Delta^{U(\sigma')}(\tau)$ for all $\tau \subseteq \sigma \cap \sigma'$.
Since different cubillages would result in different coproducts, we conclude that
\[U(\sigma)/(\sigma' \setminus \sigma) = U(\sigma \cap \sigma') = U(\sigma')/(\sigma' \setminus \sigma).\]
Thus, we have $U(\sigma) \cap {\textstyle\binom{\sigma \cap \sigma'}{i + 1}} = U(\sigma') \cap {\textstyle\binom{\sigma \cap \sigma'}{i + 1}}$.
Hence, we have
\begin{align*}
U_{\sigma} &= U \cap {\textstyle\binom{\sigma}{i + 2}} \\
&= \bigcup_{\sigma' \in \Sigma} U(\sigma') \cap {\textstyle\binom{\sigma}{i + 2}} \\
&= \bigcup_{\sigma' \in \Sigma} U(\sigma') \cap {\textstyle\binom{\sigma \cap \sigma'}{i + 2}} \\
&= \bigcup_{\sigma' \in \Sigma} U(\sigma) \cap {\textstyle\binom{\sigma \cap \sigma'}{i + 2}} = U(\sigma).
\end{align*}
We conclude that $\Delta'_{i} = \Delta_{i}^{U}$, as for all $\sigma \in \Sigma$, we have $\Delta'_{i}(\sigma) = \Delta_{i}^{U(\sigma)}(\sigma) = \Delta_{i}^{U_{\sigma}}(\sigma) = \Delta_{i}^{U}(\sigma)$.
\end{proof}


\subsection{Singular homology}\label{sect:imp:sing}

Having studied simplicial complexes, it is natural to ask for singular homology as well: can one describe all coproducts $\Delta'_{i}$ which give homotopies between $\Delta_{i - 1}^\emptyset$ and $T\Delta_{i - 1}^\emptyset$?
Here, since there are infinitely many singular simplices~$\sigma \colon \simp^n \to X$, for $X$ a topological space, it is not feasible to define coproducts differently for each singular simplex.
Making this restriction, one obtains that the only possible coproducts are the Steenrod ones.

\begin{theorem}
	\label{thm:sing}
Let $X$ be a topological space, and let $\chains_{\bullet}(X)$ denote its singular chain complex.
Suppose that we have a coproduct $\Delta'_{i}\colon \chains_{\bullet}(X) \to \chains_{\bullet}(X) \otimes \chains_{\bullet}(X)$ such that the formula for $\Delta'_{i}(\sigma)$ does not depend upon the particular singular simplex $\sigma \colon \simp^n \to X$, but only on its dimension.
Suppose further that we have \[\partial \circ \Delta'_{i} - (-1)^i \Delta'_{i} \circ \partial
 = (1 + (-1)^i T)\Delta_{i - 1}^\emptyset \ ,\]
 that $\Delta'_{i}$ has the minimal number of terms among all the coproducts satisfying this equation, and if $i=0$, we have $\Delta'_{i}(\sigma) = \sigma \otimes \sigma$ for each $0$-chain $\sigma$.
Then for all $i\geqslant 0$, we have $\Delta'_{i} = \Delta_{i}^\emptyset$ or $(-1)^{i}T\Delta_{i}^\emptyset$. 
\end{theorem}
\begin{proof}
Suppose we have a coproduct $\Delta'_{i}$ as above.
Given a singular simplex $\sigma \colon \simp^{p} \to X$, we already have that $\Delta'_{i}(\sigma) = \Delta_{i}^U(\sigma)$ for some $U \in \bruhat{[0, p]}{i + 1}$ by Theorem~\ref{thm:all_coproducts}.
We prove by induction on $p$ that in fact $U = \emptyset$ or $U  = \binom{[0, p]}{i + 2}$.
For $p \leqslant i$, we must have $\Delta'_{i}(\sigma) = 0$.
For $p = i + 1$, there are no other options for $U$, so this case is also immediate.

Suppose then that the claim holds for $p - 1$, where $p > i + 1$, so that $\Delta'_{i}(\sigma) = \Delta_{i}^{U}(\sigma)$ for $U = 0$ or $U = \binom{[0, p - 1]}{i + 1}$ for all singular $(p - 1)$-simplices $\sigma \colon \simp^{p - 1} \to X$.
Consider a singular $p$-simplex $\sigma \colon \simp^{p} \to X$.
Then, by Theorem~\ref{thm:all_coproducts}, we must have that $\Delta'_{i}(\sigma) = \Delta_{i}^{U}(\sigma)$, for some $U \in \bruhat{[0, p]}{i + 1}$.
Consequently, we must have that $\Delta'_{i}(\sigma/q) = \Delta_{i}^{U/q}(\sigma/q)$, where $\sigma/q \colon \simp^{p - 1} \to X$ are the singular simplices given by taking the face of $\sigma \colon \simp^{p} \to X$ which does not contain~$q$.
But by the induction hypothesis and the assumption that $U/q$ cannot depend upon $q$, we therefore either have $U/q = \emptyset$ for all $q$ or $U/q = \binom{[0, p] \setminus q}{i + 2}$ for all $q$.
This implies that either $U = \emptyset$ or $U = \binom{[0, p]}{i + 2}$.
This then shows that $\Delta'_{i}(\sigma) = \Delta_{i}^\emptyset(\sigma)$ or~$\Delta'_{i}(\sigma) = (-1)^{i} T\Delta_{i}^\emptyset(\sigma)$, as desired.
\end{proof}


\subsection{Reoriented higher Bruhat orders}\label{sect:imp:reorient}

In this section, we address the question of whether one can find chain homotopies between $\Delta_{i}^U$ and $T\Delta_{i}^{U}$, instead of between $\Delta_{i}^{\emptyset}$ and $T\Delta_{i}^{\emptyset}$.
The relevant posets in this case are the ``reoriented higher Bruhat orders'' of S.~Felsner and H.~Weil \cite[Section~3]{fw00}, who attribute them originally to G.~M.~Ziegler \cite{z93}.

These posets are defined as follows.
Given $U, V \in \bruhat{[0, n]}{i + 1}$, the \defn{reoriented inversion set of $V$ with respect to $U$} is defined to be the symmetric difference $ \rinv{U}{V} := (U \setminus V) \cup (V \setminus U)$.
The \defn{reoriented higher Bruhat order with respect to~$U$} $\rhbo{U}{[0, n]}{i + 1}$ has the same underlying set as $\bruhat{[0, n]}{i + 1}$, only the covering relations are given by $V \lessdot V'$ for $\rinv{U}{V'} = \rinv{U}{V} \cup \{L\}$, $L \notin \rinv{U}{V}$.
Hence, $\rhbo{U}{[0, n]}{i + 1}$ measures inversions against $U$, whereas $\bruhat{[0, n]}{i + 1}$ measures inversions against $\emptyset$.

In much the same way as in Section~\ref{sect:main}, we can construct coproducts which give homotopies between $\Delta_{i}^U$ and $T\Delta_{i}^U$ from equivalence classes of maximal chains in~$\rhbo{U}{[0, n]}{i + 1}$.
However, now that we have reoriented the higher Bruhat orders, there is a subtlety.
Note that $U$ is always minimal in $\rhbo{U}{[0, n]}{i + 1}$ whilst~$\binom{[0, n]}{i + 2} \setminus U$ is always maximal.
But there may be other minimal and maximal elements in the poset $\rhbo{U}{[0, n]}{i + 1}$ \cite[Sec.~3]{fw00}, unlike for the ordinary higher Bruhat orders.
Hence, we must restrict to maximal chains in $\rhbo{U}{[0, n]}{i + 1}$ from $U$ to $\binom{[0, n]}{i + 2} \setminus U$ to obtain a chain homotopy.

\begin{construction}\label{const:reoriented}
Let $\mathcal{W}$ be a maximal chain of $\rhbo{U}{[0, n]}{i + 1}$ from $U$ to~$\binom{[0, n]}{i + 2}~\setminus~U$.
Let further $\mathcal{W}$ be given by the sequence of faces $(F_{1}, F_{2}, \dots, F_{p})$ of~$Z([0,n],n+1)$, with $(L_{1}, L_{2}, \dots, L_{p})$ the sequence of generating vectors of these faces, with $L_{q} \in \binom{[0, n]}{i + 2}$.
The elements of the maximal chain are then $V_{q} \in \bruhat{[0, n]}{i + 1}$ for $0 \leqslant q \leqslant p$, where $\rinv{U}{V_q} = \{L_{1}, L_{2}, \dots, L_{q}\}$ and $\rinv{U}{V_0} = \emptyset$.

Recalling from \cref{const:cov_rel} the definition of a coproduct $\Delta_{i}^{F_{q}}$ given by a single face $F_{q}$ of $Z([0, n], n + 1)$, we define \[\Delta_{i + 1}^{\mathcal{W}} := \sum_{L_{q} \notin U}\Delta_{i + 1}^{F_{q}} - \sum_{L_{r} \in U}\Delta_{i + 1}^{F_{r}} \ .\]
\end{construction}

These ``reoriented'' coproducts satisfy the homotopy formula as follows.

\begin{theorem}\label{thm:reoriented}
For any maximal chain $\mathcal{W}$ of $\rhbo{U}{[0, n]}{i + 1}$ from $U$ to~$\binom{[0, n]}{i + 2}~\setminus~U$, we have that \[\partial \circ \Delta_{i + 1}^\mathcal{W} - (-1)^{i + 1} \Delta_{i + 1}^\mathcal{W} \circ \partial = (1 + (-1)^{i + 1}T)\Delta_{i}^U \ .\]
\end{theorem}
\begin{proof}
This follows from \cref{thm:cov_rel} in the same way that \cref{thm:homotopy_formula} follows from \cref{thm:cov_rel}.
The difference is that when $L_{r} \in U$, the corresponding face of $Z([0, n], n + 1)$ must be traversed in the opposite direction to the direction it is traversed in $\bruhat{[0, n]}{i + 1}$, meaning that one adds the upper facets and subtracts the lower facets, rather than \emph{vice versa}.
This is because, in terms of $\bruhat{[0, n]}{i + 1}$, such subsets need to be removed from the inversion set, rather than added.
This difference is taken care of by the extra minus signs in the definition of $\Delta_{i + 1}^\mathcal{W}$.
\end{proof}

Of course, the next question is whether one can understand the coproducts which give homotopies between $\Delta_{i + 1}^\mathcal{W}$ and $(-1)^{i + 1}T\Delta_{i + 1}^\mathcal{W}$, but it is not clear how to extend our techniques to these cases.

We end this section by illustrating Construction~\ref{const:reoriented} with an example.

\begin{example}
As in Example~\ref{ex:main}, we consider the chain complex $\chains_{\bullet}(\simp^2)$ of the $2$-simplex.
We choose $U = \{01\} \in \bruhat{[0,2]}{1}$ to reorient the higher Bruhat orders.
Hence, we wish to find a homotopy between the cup-$0$ coproduct given by \[\Delta_{0}^U([0,2]) = 1 \otimes 012 + 01 \otimes 02 + 012 \otimes 2\] and its opposite, instead of between the normal cup-$0$ coproduct $\Delta_{0}^{\emptyset}$ and its opposite~$T\Delta_{0}^{\emptyset}$.
There are two maximal chains in $\rhbo{U}{[0,2]}{1}$, namely $\mathcal{W}_{1} = (01, 12, 02)$ and $\mathcal{W}_{2} = (02, 12, 01)$, denoted by the sequences of subsets $L_{q}$ that get added to the reoriented inversion set.
Using Construction~\ref{const:reoriented}, they give the following two coproducts.
\begin{align*}
\Delta_{1}^{\mathcal{W}_{1}} &= + 01 \otimes 012 + 012 \otimes 12 - 02 \otimes 012 \ , \\
\Delta_{1}^{\mathcal{W}_{2}} &= + 012 \otimes 02 - 12 \otimes 012 - 012 \otimes 01 \ .
\end{align*}
Comparing the signs of the terms to those of Example~\ref{ex:main}, one sees that they coincide except for $01 \otimes 012$ and $012 \otimes 01$.
These are the two terms with $L = 01$, in our usual notation, which is the element of $U$.
It is routine to verify that $\Delta_{1}^{\mathcal{W}_{1}}$ and $\Delta_{1}^{\mathcal{W}_{2}}$ give homotopies from $T\Delta_{0}^U$ to $\Delta_{0}^U$.
\end{example}


\subsection{Steenrod squares}
\label{sect:squares}

Let $\Sigma$ be a finite simplicial complex.
For $i\geqslant 0$ and $U \in \bruhat{\Sigma}{i + 1}$, we denote by $\Delta_i^{U} \colon \chains_{\bullet}(\Sigma) \to \chains_{\bullet}(\Sigma) \otimes \chains_{\bullet}(\Sigma)$ the associated coproduct from \cref{const:simp_comp}.
We will consider here \defn{cochains} on $\Sigma$, which are groups 
\[
	\chains^{p}(\Sigma):= \Hom_{\mathbb{Z}}(\chains_{p}(\Sigma),\Z).
\]
Endowed with a \defn{codifferential}, a degree-$1$ map $\delta \colon \chains^{p}(\Sigma) \to \chains^{p+1}(\Sigma)$ defined by $\delta(u)(c):=u(\partial c)$, which satisfies $\delta^2=0$, they form a \defn{cochain complex} $\chains^\bullet(\Sigma)$. 

\begin{definition} \label{def:cup-i-product}
	For $i\geqslant 0$, and $U \in \bruhat{\Sigma}{i + 1}$, we define a \defn{cup-$i$ product}
\begin{align*}
\smile_i^U \colon \chains^{p}(\Sigma) \otimes \chains^{q}(\Sigma) & \rightarrow \chains^{p+q-i}(\Sigma) \\
u \otimes v &\mapsto u \smile_i^U v
\end{align*}
by the formula \[ (u \smile_i^U v) (c) := (u \otimes v)\Delta_i^U(c)\ . \]
\end{definition}

Since the coproduct $\Delta_i^U$ is a linear map, the product $\smile_i^U$ is also linear.

\begin{proposition}
	\label{prop:coboundary-formula}
	For two cochains $u \in \chains^{p}(\Sigma)$, $v \in \chains^{q}(\Sigma)$, the cup-$i$ product $\smile_i^U$ satisfies the formula
	\[
	\delta(u \smile_i^U v)
	=(-1)^i \delta u \smile_i^U v 
	+ (-1)^{i+p} u \smile_i^U \delta v 
	- (-1)^i u \smile_{i-1}^\emptyset v 
	- (-1)^{pq} v \smile_{i-1}^\emptyset u\ .
	\]
\end{proposition}
\begin{proof}
	Let $c$ be a chain of $\chains_{p+q-i+1}(\Sigma)$. 
	Using \cref{thm:simp_comp} we have
	\begin{align}
			(u \otimes v)\partial \Delta^U_{i}(c) - (-1)^i (u \otimes v)\Delta^U_{i}( \partial c) 
		&= (u \otimes v)\Delta_{i - 1}^\emptyset(c) +(-1)^i (u \otimes v)T\Delta_{i - 1}^\emptyset(c) \nonumber \\
		&= (u \otimes v)\Delta_{i - 1}^\emptyset(c) +(-1)^{i+pq} (v \otimes u)\Delta_{i - 1}^\emptyset(c) \nonumber\\
		&= (u \smile_{i - 1}^\emptyset v)(c)+(-1)^{i+pq}(v \smile_{i - 1}^\emptyset u)(c)\ . \label{eq:s1}
	\end{align}
	By definition, the two terms on the left-hand side are 
	\begin{align}
		(u \otimes v)\partial \Delta_i^U (c)
		&=[(\delta u \otimes v) + (-1)^{p}(u \otimes \delta v)]\Delta_i^U (c) \nonumber\\
		&=(\delta u \smile_i^U v + (-1)^{p} u \smile_i^U \delta v)(c)\ , \label{eq:s3}\\ 
		(u \otimes v)\Delta_i^U (\partial c)
		&=(u \smile_i^U v)(\partial c)
		=(\delta(u \smile_i^U v))(c)\ . \label{eq:s2}
	\end{align}
Substituting (\ref{eq:s2}) and (\ref{eq:s3}) into (\ref{eq:s1}), sending the term $\delta(u \smile_i^U v)$ to the right-hand side, and all the other terms to the left-hand side, and multiplying both sides by $(-1)^{i}$, we get the desired formula.
\end{proof}

So far, we have worked over $\mathbb{Z}$, with the cup-$i$ products defined on integral cochains. 
From now on, we will work over $\ZZ$.
We denote by $\chains_\bullet(\Sigma;\ZZ)$ the free $\ZZ$-module generated by the simplices of $\Sigma$, endowed with the differential $\partial(\{v_0,\ldots,v_q\}):=\sum_{p=0}^{q}\{v_0,\ldots,\hat v_p,\ldots,v_q\}$.
We denote its dual cochain complex by $\chains^{\bullet}(\Sigma;\ZZ):=\Hom_{\mathbb{Z}}(\chains_\bullet(\Sigma;\ZZ),\ZZ)$, with codifferential $\delta(u)(c):=u(\partial c)$.

\begin{lemma}\label{lem:cocycle_coboundary}
	Let $u \in \chains^{p}(\Sigma;\ZZ)$ be a cochain.
\begin{enumerate}
\item If $u$ is a cocycle modulo $2$, then $u \smile_i^U u$ is also a cocycle modulo $2$.\label{op:cocycle}
\item If $u$ is a coboundary modulo $2$, then $u \smile_i^U u$ is also a coboundary modulo~$2$.\label{op:coboundary}
\end{enumerate}
\end{lemma}
\begin{proof}
Suppose that $\delta u=2v$, for some $v \in \chains^{p+1}(\Sigma;\ZZ)$. 
By \cref{prop:coboundary-formula}, we have
\[
	\delta(u \smile_i^U u)
	= 2(v \smile_i^U u + u \smile_i^U v
	+ u \smile_{i-1}^\emptyset u)\ ,
\]
which proves Point \eqref{op:cocycle}.
Suppose now that $u=\delta w$ for some $w \in \chains^{p-1}(\Sigma;\ZZ)$.
By \cref{prop:coboundary-formula}, we have 
\begin{align*}
	\delta(w \smile_{i}^U \delta w + w \smile_{i-1}^\emptyset w)
	&=
	\delta w \smile_{i}^U \delta w + w \smile_{i}^U \delta^2 w + w \smile_{i-1}^\emptyset \delta w + \delta w \smile_{i-1}^\emptyset w \\
	&\quad + \delta w \smile_{i-1}^\emptyset w + w \smile_{i-1}^\emptyset \delta w + 2w \smile_{i-2}^\emptyset w \\
	&= \delta w \smile_{i}^U \delta w = u \smile_{i}^U u
\end{align*}
over $\ZZ$, which proves Point \eqref{op:coboundary}.
\end{proof}

\cref{lem:cocycle_coboundary} thus allows us to define the function 
\begin{align*}
		\Sq_i^U \colon  H^p(\Sigma;\ZZ) & \rightarrow  H^{2p-i}(\Sigma;\ZZ) \\
		 [u] &\mapsto [u \smile_i^U u]\ ,
\end{align*}
where $[u]$ denotes the cohomology class of a cocycle $u$.

\begin{proposition} 
\label{prop:group-homomorphism}
	Let $U \in \bruhat{\Sigma}{i+1}$ be such that $U = U' \cap \Sigma_{i + 1}$ for some $U' \in \bruhat{[0, n]}{i + 1}$.
	Then, $\Sq_i^U$ is a group homomorphism. 
\end{proposition}
\begin{proof}
As in \cref{sect:imp:reorient}, let $\mathcal{W}$ be a maximal chain in $\rhbo{U'}{[0, n]}{i + 1}$ from $U'$ to $\binom{[0, n]}{i + 2} \setminus U'$.
Let $(L_{1}, L_{2}, \dots, L_{\binom{n + 1}{i + 2}})$ be the sequence of $(i + 2)$-subsets of $[0, n]$ added to the reoriented inversion sets by~$\mathcal{W}$.
Given a simplex $\sigma \in \Sigma$, one can define $\mathcal{W}_{\sigma}$ to be the maximal chain in $\rhbo{U_{\sigma}}{\sigma}{i + 1}$ given by the subsequence of $(L_{1}, L_{2}, \dots, L_{\binom{n + 1}{i + 2}})$ contained in $\sigma$.
We then define a linear map $\Delta_{i + 1}^{\mathcal{W}} \colon \chains_{\bullet}(\Sigma;\ZZ) \to \chains_{\bullet}(\Sigma;\ZZ) \otimes \chains_{\bullet}(\Sigma;\ZZ)$ by the formula $\Delta_{i + 1}^{\mathcal{W}} := \Delta_{i + 1}^{\mathcal{W}_{\sigma}}(\sigma)$.
Combining \cref{thm:simp_comp} and \cref{thm:reoriented}, we have that $\Delta_{i + 1}^{\mathcal{W}}$ satisfies the homotopy formula with respect to $\Delta_i^U$.

Then, we define $\smile_{i + 1}^{\mathcal{W}}$ as in \cref{def:cup-i-product}, and run the proof of \cref{prop:coboundary-formula} again over $\ZZ$, showing that $\smile_{i + 1}^{\mathcal{W}}$ satisfies the formula of \cref{prop:coboundary-formula} with respect to $\smile_i^U$.
For $u$ and $v$ two cocycles in $\chains^{\bullet}(\Sigma;\ZZ)$, we obtain
\begin{align*}
		\delta (u \smile_{i+1}^{\mathcal{W}} v) &= 
		\delta u \smile_{i+1}^{\mathcal{W}} v + 
		u \smile_{i+1}^{\mathcal{W}} \delta v + u \smile_{i}^{U} v + v \smile_{i}^{U} u\\ 
		&= u \smile_i^U v + v \smile_i^U u\ .
\end{align*}
Thus, we have
\begin{align*}
		(u + v) \smile_i^U (u + v) &= u \smile_i^U u + v \smile_i^U v + u \smile_i^U v + v \smile_i^U u \\ 
		&= u \smile_i^U u + v \smile_i^U v + \delta(u \smile_{i+1}^{\mathcal{W}} v)\ ,
\end{align*}
and therefore in cohomology $\Sq_i^U(u+v)=\Sq_i^U(u)+\Sq_i^U(v)$, as desired.
\end{proof}

\begin{theorem}
	\label{thm:Steenrod-squares}
Suppose that $U, V \in \bruhat{\Sigma}{i + 1}$ are such that $U = U' \cap \Sigma_{i + 1}$ and $V = V' \cap \Sigma_{i + 1}$ for some $U', V' \in \bruhat{[0, n]}{i + 1}$.
Then, we have $\Sq_i^U = \Sq_i^V$. 
\end{theorem}

\begin{proof}
	It suffices to show this in the case where $U' \lessdot V'$ in $\bruhat{[0, n]}{i + 1}$.
	Suppose that, as in \cref{const:cov_rel}, this covering relation is given by a face $F$ of $Z([0, n], n + 1)$ with generating vectors $L$.
	If $L \notin \Sigma$, then $U = V$, so we can assume $L \in \Sigma$.
	Defining $\Delta_{i + 1}^{F}$ as in \cref{const:cov_rel} and applying \cref{thm:cov_rel}, we obtain a chain homotopy between $\Delta_{i}^{U}$ and $\Delta_{i}^{V}$.
	Indeed, note that for $\sigma \in \Sigma$ with $L \not\subseteq \sigma$, we will have $\Delta_{i}^{U}(\sigma) = \Delta_{i}^{V}(\sigma)$ and $\Delta_{i + 1}^{F}(\sigma) = 0$; for $L \subseteq \sigma$, we will have $\partial \circ \Delta_{i + 1}^{F}(\sigma) - (-1)^{i}\Delta_{i + 1}^{F} \circ \partial(\sigma) = \Delta_{i}^{U}(\sigma) - \Delta_{i}^{V}(\sigma)$ by \cref{thm:cov_rel}.
	This, in turn, gives rise to a cochain homotopy between the products $\smile_i^U$ and~$\smile_i^V$. 
	Since homotopic maps induce the same map in cohomology, we thus have $\Sq_i^U = \Sq_i^V$.
\end{proof}


\printbibliography


\appendix

\section{Sign calculations}\label{sect:appendix}

In this appendix we carry out the calculations on signs which are used in the paper, most notably in Proposition~\ref{prop:key}, but also in Proposition~\ref{prop:complement} and Lemma~\ref{lem:boundary=steenrod}.
Throughout this section, we use our convention of taking $\trmsgn(L \cup A \otimes L \cup B)$ as an element of $\mathbb{Z}/2\mathbb{Z}$ to simplify calculations.
We begin by examining how the sign changes when the halves of the tensor are swapped.

\begin{lemma}
	\label{lem:sign_swap}
Given $L, A, B \subseteq [0,n]$ which are disjoint and such that $L \cup A \cup B = [0, n]$, we have that
\begin{align*}
&\trmsgn(L \cup B \otimes L \cup A) \\
&\quad = \trmsgn(L \cup A \otimes L \cup B) + (\sz{L \cup A} + 1)(\sz{L \cup B} + 1) + \sz{L} + 1 \ .
\end{align*}
\end{lemma}
\begin{proof}
Our starting point is that
\begin{equation}\label{eq:sign:sign_swap:start}
\trmsgn(L \cup B \otimes L \cup A) = \sum_{a \in A}\sz{B}_{<a} + \sum_{l \in L}\sz{L}_{<l} + (n + 1)\sz{B}.
\end{equation}
We then make the following calculations, in which we use the fact that $n + 1 = \sz{L \cup A \cup B}$ and $x = x^2$ in~$\mathbb{Z}/2\mathbb{Z}$.
\begin{align*}
(n + 1)\sz{B} &= \sz{L \cup A \cup B}\sz{A} + \sz{L \cup A \cup B}\sz{B} + (n + 1)\sz{A} \\
&= \sz{L}\sz{A} + \sz{A} + \sz{B}\sz{A} + \sz{L}\sz{B} + \sz{A}\sz{B} + \sz{B} + (n + 1)\sz{A} \\
&= \sz{L \cup A}\sz{L \cup B} + \sz{A} + \sz{B} + \sz{L} + \sz{A}\sz{B} + (n + 1)\sz{A} \\
&= \sz{L \cup A}\sz{L \cup B} + \sz{L \cup A} + \sz{L \cup B} + \sz{L} + \sz{A}\sz{B} + (n + 1)\sz{A} \\
&= (\sz{L \cup A} + 1)(\sz{L \cup B} + 1) + \sz{L} + 1 + \sz{A}\sz{B} + (n + 1)\sz{A} \ .
\end{align*}
From this and \eqref{eq:sign:sign_swap:start}, noting that
\[\sum_{a \in A}\sz{B}_{<a} = \sum_{b \in B}\sz{A}_{>b} = \sum_{b \in B}(\sz{A}_{<b} + \sz{A}) = \sum_{b \in B}\sz{A}_{<b} + \sz{A}\sz{B} \ ,\]
we obtain that
\begin{align*}
\trmsgn(L \cup B \otimes L \cup A) &= \sum_{b \in B}\sz{A}_{<b} + \sum_{l \in L}\sz{L}_{<l} + (n + 1)\sz{A} \\
&\quad + (\sz{L \cup A} + 1)(\sz{L \cup B} + 1) + \sz{L} + 1 \ ,
\end{align*}
from which the result follows.
\end{proof}

We now compare our signs with those associated to the overlapping partitions.

\begin{lemma}
	\label{lem:sign:steenrod}
Given $L = \{l_{0}, l_{1}, \dots, l_{i}\} \in \binom{[0, n]}{i + 1}$, let $\mathcal{L}$ be the overlapping partition $[0,l_{0}], [l_{0}, l_{1}], \dots, [l_{i},n]$.
Further, let $A = [0, l_{0}) \cup (l_{1}, l_{2}) \cup \dots$ and $B = (l_{1}, l_{2}) \cup (l_{3}, l_{4}) \cup \dots$. 
Then we have
\[\trmsgn(\mathcal{L}) = \trmsgn(L \cup A \otimes L \cup B) + \sum_{l \in L} \sz{L}_{<l} + \sz{L} + 1 \ .\]
Hence, $\trmsgn(\mathcal{L}) = \trmsgn(L \cup A \otimes L \cup B) + \floor{i/2}$.
\end{lemma}
\begin{proof}
By definition, the sign associated to the overlapping partition is
\begin{equation}
\label{eq:sign:steenrod:def}
\trmsgn(\mathcal{L}) = \mathsf{sign}(w_{\mathcal{L}}) + in,
\end{equation}
where $w_{\mathcal{L}}$ is the permutation defined in Section~\ref{sect:back:steenrod_co}. One can calculate the sign associated with the permutation $w_{\mathcal{L}}$ as 
\begin{align*}
\mathsf{sign}(w_{\mathcal{L}}) &= \sum_{b \in B}\sz{A \cup L}_{>b} \\
&= \sum_{b \in B} \sz{A}_{>b} + \sum_{b \in B}\sz{L}_{>b} \\
&= \sum_{b \in B} \sz{A}_{<b} + \sz{A}\sz{B} + \sum_{b \in B}\sz{L}_{>b} \ .
\end{align*}
We then note from Lemma~\ref{lem:boundary=steenrod} that $B$ consists of the odd gaps in $L$ for $\sz{L}$ even, and even gaps in $L$ for $\sz{L}$ odd.
Hence, for all $b \in B$, we have that $\sz{L}_{>b} = \sz{L} + 1$ in $\mathbb{Z}/2\mathbb{Z}$, and so
\begin{equation}\label{eq:sign:steenrod:perm}
\mathsf{sign}(w_{\mathcal{L}}) = \sum_{b \in B} \sz{A}_{<b} + \sz{A}\sz{B} + \sz{B}(\sz{L} + 1).
\end{equation}
We then calculate that
\begin{align*}
in &= (\sz{L} + 1)n \\
 &= (\sz{L} + 1)(n + 1) + \sz{L} + 1 \\ 
 &= (\sz{L} + 1)\sz{A \cup B} + \sz{L}^2 + \sz{L} + \sz{L} + 1 \\
 &= (\sz{L} + 1)\sz{A \cup B} + \sz{L} + 1 \ .
\end{align*}
Putting this together with \eqref{eq:sign:steenrod:def} and \eqref{eq:sign:steenrod:perm}, we obtain
\begin{align*}
\trmsgn(\mathcal{L}) &= \sum_{b \in B} \sz{A}_{<b} + \sz{A}\sz{B} + \sz{B}(\sz{L} + 1) + (\sz{L} + 1)\sz{A \cup B} + \sz{L} + 1 \\
&= \sum_{b \in B} \sz{A}_{<b} + \sz{A}\sz{B} + (\sz{L} + 1)\sz{A} + \sz{L} + 1 \\
&= \trmsgn(L \cup A \otimes L \cup B) + \sum_{l \in L} \sz{L}_{<l} + (n + 1)\sz{A} + \sz{A}\sz{B} + (\sz{L} + 1)\sz{A} + \sz{L} + 1 \\
&= \trmsgn(L \cup A \otimes L \cup B) + \sum_{l \in L} \sz{L}_{<l} + \sz{L \cup A}\sz{A} + (\sz{L} + 1)\sz{A} + \sz{L} + 1 \\
&= \trmsgn(L \cup A \otimes L \cup B) + \sum_{l \in L} \sz{L}_{<l} + \sz{A}\sz{A} + \sz{A} + \sz{L} + 1 \\
&= \trmsgn(L \cup A \otimes L \cup B) + \sum_{l \in L} \sz{L}_{<l} + \sz{L} + 1 \ .
\end{align*}
Since $\sz{L} = i + 1$, one can then straightforwardly verify that \[\sum_{l \in L} \sz{L}_{<l} + \sz{L} + 1 = i/2 \in \mathbb{Z}/2\mathbb{Z}\] for $i$ even and $\floor{i/2}$ for $i$ odd.
\end{proof}

We now carry out the sign calculations which feed into Proposition~\ref{prop:key}.

\begin{lemma}
	\label{lem:sign:L_first}
For $k \in L$, we have that
\[\trmsgn(L \cup A \otimes L \cup B) + \sz{L \cup A}_{<k} = \trmsgn((L \setminus k) \cup A \otimes (L \setminus k) \cup (B \cup k)) + \sz{L}_{>k} \ .\]
\end{lemma}
\begin{proof}
We note that \[\sum_{b \in B}\sz{A}_{<b} = \sum_{b \in B \cup k}\sz{A}_{<b} + \sz{A}_{<k}\] and \[\sum_{l \in L}\sz{L}_{<l} = \sum_{l \in L\setminus k} \sz{L}_{<l} + \sz{L}_{<k} \ .\]
Using these, we obtain from the definition of $\trmsgn(L \cup A \otimes L \cup B)$ that
\begin{align*}
\trmsgn(L \cup A \otimes L \cup B) + \sz{L \cup A}_{<k} &= \sum_{b \in B \cup k} \sz{A}_{<b} + \sum_{l \in L\setminus k}\sz{L}_{<l} + (n + 1)\sz{A} \\
&= \sum_{b \in B \cup k} \sz{A}_{<b} + \sum_{l \in L\setminus k}\sz{L\setminus k}_{<l} + \sz{L}_{>k} + (n + 1)\sz{A} \\
&= \trmsgn((L \setminus k) \cup A \otimes (L \setminus k) \cup (B \cup k)) + \sz{L}_{>k} \ .
\end{align*}
\end{proof}

\begin{lemma}
	\label{lem:sign:L_second}
For $k \in L$, we have that
\begin{align*}
&\trmsgn(L \cup A \otimes L \cup B) + \sz{L \cup B}_{<k} + \sz{L \cup A} + 1 \\
&\quad = \trmsgn((L \setminus k) \cup (A \cup k) \otimes (L\setminus k) \cup B) + \sz{L}_{>k} + 1 \ .
\end{align*}
\end{lemma}
\begin{proof}
We break down this calculation into the following three steps.
\begin{align*}
\sum_{b \in B} \sz{A}_{<b} &= \sum_{b \in B} \sz{A \cup k}_{<b} + \sz{B}_{>k} \ , \\
\sum_{l \in L} \sz{L}_{<l} &= \sum_{l \in L \setminus k} \sz{L}_{<l} + \sz{L}_{<k} \ , \\
&= \sum_{l \in L \setminus k} \sz{L \setminus k}_{<l} + \sz{L}_{>k} + \sz{L}_{<k} \ , \\
(n + 1)\sz{A} + \sz{L \cup A} &= (n + 1)\sz{A \cup k} + (n + 1) + \sz{L \cup A} \\
&= (n + 1)\sz{A \cup k} + \sz{B} \ .
\end{align*}
Using these, we obtain that
\begin{align*}
&\trmsgn(L \cup A \otimes L \cup B) + \sz{L \cup A} \\
&\quad = \trmsgn((L \setminus k) \cup (A \cup k) \otimes (L\setminus k) \cup B) + \sz{B}_{>k} + \sz{L}_{<k} + \sz{L}_{>k} + \sz{B} \\
&\quad = \trmsgn((L \setminus k) \cup (A \cup k) \otimes (L\setminus k) \cup B) + \sz{B}_{<k} + \sz{L}_{<k} + \sz{L}_{>k} \ ,
\end{align*}
and hence the result follows.
\end{proof}

\begin{lemma}
	\label{lem:sign:A}
For $k \in A$, we have that
\begin{align*}
&\trmsgn(L \cup A \otimes L \cup B) + \sz{L \cup A}_{<k} \\
&\quad = \trmsgn(L \cup (A \setminus k) \otimes L \cup B) + \sz{L} + \sz{L \cup A \cup B}_{<k} + 1 \ .
\end{align*}
\end{lemma}
\begin{proof}
Here we note that
\[\sum_{b \in B} \sz{A}_{<b} = \sum_{b \in B} \sz{A \setminus k}_{<b} + \sz{B}_{>k}\]
and
\begin{align*}
(n + 1)\sz{A} &= n\sz{A} + \sz{A} \\
&= n\sz{A \setminus k} + n + \sz{A} \\
&= n\sz{A \setminus k} + \sz{L} + \sz{B} + 1 \ .
\end{align*}
Using these two facts, we deduce that
\begin{align*}
&\trmsgn(L \cup A \otimes L \cup B) + \sz{L \cup A}_{<k} \\
&\quad = \trmsgn(L \cup (A \setminus k) \otimes L \cup B) + \sz{B}_{>k} + \sz{L} + \sz{B} + 1 + \sz{L \cup A}_{<k} \\
&\quad = \trmsgn(L \cup (A \setminus k) \otimes L \cup B) + \sz{L} + \sz{L \cup A \cup B}_{<k} + 1 \ .
\end{align*}
\end{proof}

\begin{lemma}
	\label{lem:sign:B}
For $k \in B$, we have that
\begin{align*}
&\trmsgn(L \cup A \otimes L \cup B) + \sz{L \cup A} + \sz{L \cup B}_{<k} + 1 \\
&\quad = \trmsgn(L \cup A \otimes L \cup (B\setminus k)) + \sz{L} + \sz{L \cup A \cup B}_{<k} + 1 \ .
\end{align*}
\end{lemma}
\begin{proof}
Here we use that
\[\sum_{b \in B} \sz{A}_{<b} = \sum_{b \in B \setminus k} \sz{A}_{<b} + \sz{A}_{<k}\]
and $(n + 1)\sz{A} + \sz{L \cup A} = n\sz{A} + \sz{L}$.
This allows us to deduce that
\begin{align*}
&\trmsgn(L \cup A \otimes L \cup B) + \sz{L \cup A} + \sz{L \cup B}_{<k} + 1 \\
&\quad = \trmsgn(L \cup A \otimes L \cup (B\setminus k)) + \sz{A}_{<k} + \sz{L} + \sz{L \cup B}_{<k} + 1 \\
&\quad = \trmsgn(L \cup A \otimes L \cup (B\setminus k)) + \sz{L} + \sz{L \cup A \cup B}_{<k} + 1 \ .
\end{align*}
\end{proof}

\end{document}